\documentclass[a4paper,11pt]{article}
\title{\bf{Birational Calabi-Yau 3-folds and BPS state counting}
}
\date{}
\author{Yukinobu Toda}

\usepackage{makeidx}

\usepackage{latexsym}
\usepackage{amscd}
\usepackage{amsmath}
\usepackage{amssymb}
\usepackage{amsthm}
\usepackage{float}
\usepackage[all]{xy}
\usepackage{graphicx}

\usepackage{array}
\usepackage{amscd}
\usepackage[all]{xy}
\usepackage{makeidx}
\usepackage{latexsym}
\DeclareFontFamily{U}{rsfs}{%
\skewchar\font127}
\DeclareFontShape{U}{rsfs}{m}{n}{%
<-6>rsfs5<6-8.5>rsfs7<8.5->rsfs10}{}
\DeclareSymbolFont{rsfs}{U}{rsfs}{m}{n}
\DeclareSymbolFontAlphabet
{\mathrsfs}{rsfs}
\DeclareRobustCommand*\rsfs{%
\@fontswitch\relax\mathrsfs}
\theoremstyle{plain}
\newtheorem{thm}{Theorem}[section]
\newtheorem{prop}[thm]{Proposition}
\newtheorem{lem}[thm]{Lemma}
\newtheorem{sublem}[thm]{Sublemma}
\newtheorem{defi}[thm]{Definition}
\newtheorem{rmk}[thm]{Remark}
\newtheorem{cor}[thm]{Corollary}

\newtheorem{prop-defi}[thm]{Proposition-Definition}

\newtheorem{question}[thm]{Question}

\newtheorem{conj}[thm]{Conjecture}
\newtheorem{exam}[thm]{Example}

\newcommand{\aA}{\mathcal{A}}

\newcommand{\cC}{\mathcal{C}}
\newcommand{\dD}{\mathcal{D}}
\newcommand{\eE}{\mathcal{E}}
\newcommand{\fF}{\mathcal{F}}

\newcommand{\hH}{\mathcal{H}}
\newcommand{\iI}{\mathcal{I}}

\newcommand{\mM}{\mathcal{M}}
\newcommand{\nN}{\mathcal{N}}
\newcommand{\oO}{\mathcal{O}}
\newcommand{\pP}{\mathcal{P}}
\newcommand{\qQ}{\mathcal{Q}}
\newcommand{\rR}{\mathcal{R}}
\newcommand{\sS}{\mathcal{S}}
\newcommand{\tT}{\mathcal{T}}
\newcommand{\uU}{\mathcal{U}}

\newcommand{\wW}{\mathcal{W}}
\newcommand{\xX}{\mathcal{X}}
\newcommand{\yY}{\mathcal{Y}}
\newcommand{\zZ}{\mathcal{Z}}
\newcommand{\lr}{\longrightarrow}

\newcommand{\Supp}{\mathop{\rm Supp}\nolimits}
\newcommand{\Hom}{\mathop{\rm Hom}\nolimits}

\newcommand{\dR}{\mathbf{R}}

\newcommand{\Var}{\mathop{\rm Var}\nolimits}
\newcommand{\SF}{\mathop{\rm SF}\nolimits}

\newcommand{\Chow}{\mathop{\rm Chow}\nolimits}

\newcommand{\id}{\textrm{id}}

\newcommand{\ch}{\mathop{\rm ch}\nolimits}

\newcommand{\Ext}{\mathop{\rm Ext}\nolimits}
\newcommand{\Spec}{\mathop{\rm Spec}\nolimits}

\newcommand{\Coh}{\mathop{\rm Coh}\nolimits}

\newcommand{\cneq}{\mathrel{\raise.095ex\hbox{:}\mkern-4.2mu=}}
\newcommand{\eqcn}{\mathrel{=\mkern-4.5mu\raise.095ex\hbox{:}}}

\newcommand{\NE}{\mathop{\rm NE}\nolimits}

\newcommand{\Aut}{\mathop{\rm Aut}\nolimits}

\newcommand{\Stab}{\mathop{\rm Stab}\nolimits}

\newcommand{\oPPer}{\mathop{\rm ^{0}Per}\nolimits}
\newcommand{\ofPPer}{\mathop{\rm ^{0}\mathfrak{Per}}\nolimits}

\newcommand{\iPPer}{\mathop{\rm ^{-1}Per}\nolimits}
\newcommand{\ifPPer}{\mathop{\rm ^{-1}\mathfrak{Per}}\nolimits}

\newcommand{\fPPer}{\mathop{\rm ^{\mathit{p}}\mathfrak{Per}}\nolimits}

\newcommand{\pPPer}{\mathop{\rm ^{\mathit{p}}Per}\nolimits}

\newcommand{\Imm}{\mathop{\rm Im}\nolimits}
\newcommand{\Ree}{\mathop{\rm Re}\nolimits}

\newcommand{\GL}{\mathop{\rm GL}\nolimits}

\newcommand{\length}{\mathop{\rm length}\nolimits}
\newcommand{\St}{\mathop{\rm St}\nolimits}

\begin{document}
\maketitle

\begin{abstract}
This paper contains some applications of Bridgeland-Douglas
stability conditions 
on triangulated categories, and Joyce's work on 
counting invariants of semistable objects, 
to the study of 
birational geometry. 
 We introduce the notion of motivic Gopakumar-Vafa invariants 
 as counting invariants of D2-branes, and show that 
 they are invariant under 
 birational transformations between Calabi-Yau 3-folds. 
The result is similar to the fact that birational 
Calabi-Yau 3-folds have the same 
betti numbers or Hodge numbers. 

\end{abstract}

\section{Introduction}
First of all, let us recall the following well-known fact. 
\begin{thm}\emph{\bf{\cite{Kof}}} \label{first}
Let $\phi \colon W\dashrightarrow X$ be a birational 
map between smooth projective Calabi-Yau 3-folds. Then we have 
$$b_i(W)=b_i(X) \quad \mbox{ for all }i\in \mathbb{Z}.$$
Here $b_i(\ast)$ is the $i$-th betti number of $\ast$. 
\end{thm}
This result has been generalized for all dimensions by~\cite{Bat}, \cite{Wang}
using the method of \textit{$p$-adic integration}.
Later on the equality of Hodge numbers $h^{p,q}(X)$
(more generally stringy Hodge numbers $h_{st}^{p,q}(X)$ for 
varieties with log terminal singularities)
under birational maps
has been proved in~\cite{KonM}, \cite{Bat2}, \cite{Den} using the 
method of \textit{motivic integration}. 
(Also see~\cite{Yasu}, \cite{Ito} for related works.)

In terms of string theory, the numbers $b_i(X)$, $h^{p,q}(X)$
are interpreted as ``counting invariants" of BPS
D0-branes, which are mathematically stable zero-dimensional 
sheaves $\{\oO _x\}_{x\in X}$. 
 In this paper, we shall 
address the following question. 

\begin{question}\label{quest}\emph{
How do counting invariants of BPS D2-branes (i.e. stable 
one dimensional sheaves)
transform under 
birational transformations ?}
\end{question}

A similar problem has been studied in~\cite{Morr}, \cite{LiRu}, \cite{LY}, 
\cite{Komi}
for Gromov-Witten invariants and in~\cite{HL}, \cite{Sz}
for Donaldson-Thomas invariants. 
In this paper, we 
are interested in Question~\ref{quest}
for Gopakumar-Vafa invariants, 
which were
introduced by physicists Gopakumar and Vafa~\cite{GV}. 
In the paper~\cite{HST},  
Hosono, Saito and Takahashi~\cite{HST} proposed a first mathematical 
formulation of them. 
The purpose of this paper is to introduce another mathematical formulation
which we call \textit{motivic Gopakumar-Vafa invariants}
and study their behavior under birational transformations. 
Our method is quite different from the above works, 
and uses \textit{Bridgeland-Douglas stability conditions on 
triangulated categories}~\cite{Brs1}, \cite{Dou1}, \cite{Dou2},
 and \textit{Joyce's counting invariants of semistable objects}~\cite{Joy1},
 \cite{Joy2}, \cite{Joy3}, \cite{Joy4}.

\subsection{Gopakumar-Vafa invariants}
Let $X$ be a Calabi-Yau 3-fold over $\mathbb{C}$,  
$\beta \in N_1(X)$ and $g\in \mathbb{Z}_{\ge 0}$,
where $N_1(X)$ is the $\mathbb{R}$-vector space 
of numerical classes of one cycles. 
 The 0-point genus $g$ Gromov-Witten invariant of $X$ in 
 numerical class $\beta$ is defined as the integration over the 
 virtual fundamental class of the moduli space of stable maps
 $\overline{\mM}_{g,0}(X, \beta)$, 
 $$N_{g}^{\beta} \cneq \int _{\overline{\mM}_{g,0}^{virt}(X, \beta)}
 1_{\overline{\mM}_{g,0}(X,\beta)} \in \mathbb{Q}.$$
 Although the invariants $N_{g}^{\beta}$ are not integers 
 in general, Gopakumar and Vafa~\cite{GV} claimed the following integrality 
 of the generating function involving $N_{g}^{\beta}$, based on the string 
 duality between Type IIA and M-theory. 
 \begin{conj}\label{conj}
  There are integers $n_{g}^{\beta,conj}$ and 
  the following equality in  
  $\mathbb{Q}[[N_1(X),\lambda^{\pm 1}]]$,
 $$\sum _{\beta, g}
 N_g ^{\beta}q^{\beta}\lambda^{2g-2}= 
  \sum _{\beta, g, k}\frac{n_{g}^{\beta,conj}}{k}\left( 2\sin \left(
 \frac{k\lambda}{2} \right) \right)^{2g-2}q^{k\beta}.$$
 \end{conj}
 They also claimed that the integer $n_{g}^{\beta,conj}$ 
 should be defined by 
 the ``virtual counting of genus $g$ Jacobians'' in
 the moduli space of the BPS-branes 
 wrapping  
 around holomorphic curves in $X$, 
 and some computations are done in~\cite{KKV}.
 Its mathematical proposal by Hosono, Saito
 and Takahashi~\cite{HST} uses the relative Lefschetz action 
 on the intersection cohomology of the moduli space 
 of one dimensional semistable sheaves $E$ with 
 \begin{align}\label{ch}(\ch_2(E), \ch _3(E))=(\beta, 1), \end{align}
 which we denote by $M^{\beta}$. 
 Using an $\mathfrak{sl}_2\times \mathfrak{sl}_2$
 -action on $IH^{\ast}(\widetilde{M}^{\beta})$ where 
  $\widetilde{M}^{\beta} \to M^{\beta}$ is the normalization,  
 they defined an invariant
 $\tilde{n}_g^{\beta} \in \mathbb{Z}$
 and conjectured that the invariants $\tilde{n}_g^{\beta}$ 
 satisfy Conjecture~\ref{conj}. 
 
 However it seems that the invariants $\tilde{n}_g^{\beta}$ 
 are unlikely 
 to be deformation invariant, hence not 
 equal to $n_{g}^{\beta, cong}$ exactly, since 
the definition of $\tilde{n}_g^{\beta}$ 
 does not involve virtual classes.  
 Now there is another approach of Gopakumar-Vafa invariants 
 using the notion of stable pairs, 
 proposed by Pandharipande and Thomas~\cite{PT}.

 \subsection{Main result}
 Based on the work~\cite{HST}, we will 
 construct invariants, (cf. Definition~\ref{mgd})
 $$n_g^{\beta}(X) \in \mathbb{Z}, \quad \mbox{ for }
 g\ge 0, \ \beta \in N_1(X),$$
 as a refinement of $\tilde{n}_g^{\beta}$. Roughly 
 speaking $n_g^{\beta}(X)$ is defined using  
 a certain motivic invariant, similar to the virtual Poincar\'e
 polynomial of $M^{\beta}$. 
 Furthermore $n_g^{\beta}(X)$ is also defined 
 for a non-effective one cycle class $\beta$. 
 At least $n_g^{\beta}(X)$ 
 coincides with $\tilde{n}_g^{\beta}$
 if the moduli space $M^{\beta}$ is smooth
 and $\beta$ is represented by an effective one cycle.   
 The following is our main theorem. 
 \begin{thm}\label{result}
 Let $\phi \colon W\dashrightarrow X$ be a birational 
 map between smooth projective Calabi-Yau 3-folds. 
 Then we have 
 $$n_{g}^{\beta}(W)=n_{g}^{\phi_{\ast}\beta}(X),$$
 for all $g\ge 0$ and $\beta \in N_1(W)$. 
 \end{thm}
 It is worth mentioning that 
 in the proof of Theorem~\ref{result}, 
 putting $\beta =0$ would 
 result Theorem~\ref{first}. 
 (cf. Remark~\ref{betti}.)
 The definition of $n_g^{\beta}$ 
 also does not involve virtual classes, so we do not claim that
 $n_g^{\beta}$ satisfy Conjecture~\ref{conj}. 
  However we have obtained 
  a certain mathematical approximation of Gopakumar-Vafa invariants, 
 which have birational invariance property.  
 
 \subsection{Strategy of the proof of Theorem~\ref{result}}
 We use the notion of stability conditions on 
 triangulated categories introduced by T.~Bridgeland~\cite{Brs1},
 based on M.~Douglas's work on $\Pi$-stability~\cite{Dou1}, \cite{Dou2}. 
 Roughly speaking a stability condition on 
 a triangulated category $\dD$ consists of data 
 $$Z\colon K(\dD) \lr \mathbb{C}, \quad \pP(\phi)\subset \dD,$$
 where $Z$ is a group homomorphism and $\pP(\phi)$ is a subcategory 
 for each $\phi \in \mathbb{R}$, which satisfy
 some axioms. (cf. Definition~\ref{stde}.) 
 We work over the triangulated category $\dD =\dD_X$
 defined by 
 $$\dD _X \cneq \{ E\in D^b(\Coh(X)) \mid \dim \Supp(E)\le 1\}.$$ 
 In terms of string theory, the set of objects 
 $\{\pP(\phi)\mid \phi \in \mathbb{R}\}$
 is supposed to represent the set of BPS-branes at some 
 point of the so called 
 \textit{stringy K$\ddot{\textrm{a}}$hler moduli space}, 
 the subspace of the moduli space of $\nN=2$ super conformal 
 field theories.  
 Indeed Bridgeland~\cite{Brs1} showed that the 
 set of good stability conditions $\Stab(X)$ is a complex 
 manifold, and expected that it describes the 
 stringy K$\ddot{\textrm{a}}$hler moduli space mathematically. 
 In this paper, we will construct a connected open subset, (cf. Lemma~\ref{op})
 $$U_X \subset \Stab(X),$$
 which corresponds to the \textit{neighborhood of the large 
 volume limit} in string theory. 
 Then our invariant $n_g^{\beta}(X)$ is defined as a
 certain counting invariant of objects $E\in \pP(\phi)$ 
 which satisfy (\ref{ch}) with respect to some 
 point $(Z,\pP)\in U_X$. 
 
 Next let us consider a birational map $\phi \colon W\dashrightarrow X$
 from another Calabi-Yau 3-fold $W$. Then due to Bridgeland~\cite{Br1}, 
 we have the equivalence of triangulated categories, 
 $$\Phi \colon \dD _W \lr \dD _X, $$
 which gives an isomorphism $\Phi _{\ast}\colon \Stab(W) \to \Stab(X)$. 
 We claim that the closures of
 $\Phi_{\ast}U_W$ and $U_X$ intersect, in particular they
 are contained in the same 
 connected component of $\Stab(X)$. (cf. Lemma~\ref{pt}.)
 Then Question~\ref{quest} is rephrased as follows:
 \begin{question}\label{quest2}\emph{
 How do counting invariants vary by changing stability 
 conditions, from $\sigma \in U_X$ to 
 $\tau \in \Phi_{\ast}U_W$ ? }
 \end{question}
 Now we use D.~Joyce's theory on configurations on
 abelian categories and counting invariants of 
 semistable objects~\cite{Joy1}, \cite{Joy2}, \cite{Joy3}, \cite{Joy4}.
 Especially in~\cite{Joy4}, he studies how counting invariants 
 of semistable objects
 vary 
 under change of stability conditions.
  Although his works focus on 
 stability conditions on abelian categories, his 
 arguments also apply in our case. 
 The reason is as follows: roughly speaking, a theory of stability 
 conditions on abelian categories corresponds to a 
 local theory on $\Stab(X)$. Thus Joyce's works enable us to
  study how counting invariants vary 
 locally, and actually we will see they do not vary at all.  
 Obviously we can answer Question~\ref{quest2}, 
 and conclude Theorem~\ref{result} by this argument. 
 
 The content of this paper is as follows. In Section~\ref{review}
 we review the mathematical proposal of 
 Gopakumar-Vafa invariants in~\cite{HST}. 
 In Section~\ref{Stab} we review Bridgeland's work 
 on stability conditions on triangulated categories~\cite{Brs1}, 
 and construct some stability conditions we need. 
 In Section~\ref{Mot} we introduce our invariant 
 $n_g^{\beta}(X)$, and prove Theorem~\ref{result} 
 in Section~\ref{Bir}. In Section~\ref{tech} we 
 prove some technical lemmas. 
 
 \subsection{Terminology used in this paper}
 In this paper, all the varieties are defined over 
 $\mathbb{C}$. 
 We say $X$ is Calabi-Yau if $X$ is smooth projective with 
 trivial canonical bundle. 
 For a variety $X$, we denote 
 by $D(X)$ the derived category of coherent sheaves on $X$. 
 For a triangulated category $\dD$, its $K$-group is 
 denoted by $K(\dD)$. 
 We use the following standard terminology
  used in birational geometry~\cite{KM}, 
 $$
 N^1(X)\cneq \bigoplus _{D\subset X}\mathbb{R}D/\equiv, \quad
 N_1(X)\cneq\bigoplus_{C\subset X} \mathbb{R}C/\equiv.$$
 In the definition of $N^1(X)$, 
  $D\subset X$ is a divisor and $D_1 \equiv D_2$
  if and only if $D_1\cdot C=D_2 \cdot C$ for any curve $C$ on $X$. 
  Similarly in the definition of $N_1(X)$, $C$ is a curve on $X$ and 
  $C_1 \equiv C_2$ if and only if $D\cdot C_1 =D\cdot C_2$
  for any divisor $D$. 
  Cleary we have the perfect pairing, 
  $$N^1(X) \times N_1(X) \ni (D,C) \longmapsto D\cdot C\in \mathbb{R},$$
  which identifies $N_1(X)$ with the dual of $N^1(X)$. 
  We set $N^1(X)_{\mathbb{C}}=N^1(X)\otimes _{\mathbb{R}}\mathbb{C}$ and 
  \begin{align*}
  \overline{NE}(X)&\cneq \overline{\{\mbox{Cone of effective curves}\}}\subset
  N_1(X), \\
  A(X)_{\mathbb{C}}&\cneq \{ B+i\omega \in N^1(X)_{\mathbb{C}} \mid 
  \omega \mbox{ is ample }\} \subset N^1(X)_{\mathbb{C}}. 
  \end{align*}  
  Suppose that a birational map
   $\phi \colon W\dashrightarrow X$ is an isomorphism in codimension
   one. We use the following isomorphisms, 
   $$\phi _{\ast}\colon N^1(W) \lr N^1(X), \quad \phi_{\ast}\colon 
   N_1(W) \lr N_1(X),$$
   where the LHS is the strict transform and the RHS
   is the inverse of the dual of the LHS. 
 For a non-zero $\beta \in N_1(X)$, 
 $\Chow_{\beta}(X)$ is the subvariety of the Chow variety 
 $\Chow (X)$,  
 representing effective one cycles on $X$ with numerical class $\beta$. 
 One can refer~\cite[Chapter 1, Section 3]{Ko} for the existence of the variety $\Chow_{\beta}(X)$. 
 We set $\Chow_{\beta}(X)=\Spec \mathbb{C}$ when $\beta =0$. 
 For a coherent sheaf $E$ on $X$ with 
 $\dim \Supp(E)\le 1$, we set
 \begin{align}\label{cycle}
 s(E)\cneq 
  \sum _{p\in X}\length _{\oO _{X,p}}(E _p)\overline{\{p\}} 
   \in \Chow _{\beta}(X),\end{align}
   where $\beta=\ch _2(E)$ and
   $p$ runs through all the codimension two points.
   
   Let $X$ be a Calabi-Yau 3-fold. 
   For an object $E\in \dD _X$ and $v=(\beta, k)\in N_1(X)\oplus \mathbb{Z}$, 
   we say $E$ is of numerical type $v$ if 
   $$(\ch _2(E), \ch _3(E))=(\beta, k).$$

 \section{Review of work of Hosono, Saito and 
 Takahashi}\label{review}
 In this section, we briefly review the work of
 Hosono, Saito and Takahashi~\cite{HST}. 
 \subsection{Representations of $\mathfrak{sl}_2$}\label{rep}
 First let us recall that the Lie algebra $\mathfrak{sl} _2$ is generated by 
 three elements, 
 $$e=
 \left(\begin{array}{cc}0 & 1 \\ 0 & 0 \end{array}\right), \quad 
 f=\left(\begin{array}{cc}0 & 0 \\ 1 & 0 \end{array}\right), \quad
 h=\left(\begin{array}{cc}1& 0 \\ 0 & -1 \end{array}\right),$$
 which satisfy the relation, 
 $$[e, f]=h, \quad [h, e]=2e, \quad [h, f]=-2f.$$
 For each $j\in \frac{1}{2}\mathbb{Z}$, there is a unique 
 irreducible representation of $\mathfrak{sl}_2$ (up to isomorphism) of 
 dimension $2j+1$, called the \textit{spin $j$ representation},
  and denoted by $(j)$. 
  For $V=(j)$, there is an 
  eigenvector $v\in V$ of $h$ such that $fv=0$ and 
  $$(j)=<v, ev, \cdots, e^{2j}v>, \quad e^{2j+1}v=0,$$
  with $he^k v=(-2j+2k)v$, $0\le k\le 2j$.  
  
  Let $X$ be a normal projective variety and $IH^{\ast}(X)$
  is the intersection cohomology of $X$ introduced in~\cite{BBD}.
   Note that if $X$ is connected and smooth, 
  we have 
  \begin{align*}
  IH^i(X)= H^{i+\dim X}(X, \mathbb{C}),\end{align*}
  for any $i\in \mathbb{Z}$. 
  Let $H$ be an ample divisor on $X$, and $\eta$ be the 
  Lefschetz operator, 
  $$  \eta= H \wedge  
  \colon IH^{\ast}(X) \to IH^{\ast +2}(X).$$
  It is well-known that the operator $\eta ^i \colon IH^{-i}(X)\to 
  IH^{i}(X)$ is an isomorphism. (cf.~\cite[Theorem 5.4.10, 6.2.10]{BBD}.)
   Using this, one can construct
  an $\mathfrak{sl} _2$-action on $IH^{\ast}(X)$ as follows. 
  (See~\cite[Section 2]{SS} for the detail.)
  First we find a homogeneous basis of $IH^{\ast}(X)$ which 
  consist of primitive elements. 
   Here $v\in IH^{-i}(X)$ for $i\ge 0$ is primitive 
  if $\eta ^{i+1}v=0$. 
  For such a basis $\{v_{\alpha}\}$,   
 $IH^{\ast}(X)$ is written as a direct sum 
  of the subspaces generated by 
  \begin{align}\label{spin}
  \{v_{\alpha}, \eta v_{\alpha}, \cdots, 
  \eta^{i_{\alpha}}v_{\alpha}\},\end{align}
   with $\deg v_{\alpha}=-i_{\alpha}$. 
  Then define the representation of $\mathfrak{sl}_2$ on $IH^{\ast}(X)$ 
  by letting $e\in \mathfrak{sl}_2$ act as $\eta$, $h\in \mathfrak{sl}_2$ act as multiplying 
  by the degree and the action of $f\in \mathfrak{sl}_2$ is defined 
  inductively from the requirement of $fv_{\alpha}=0$. 
  Hence the subspace (\ref{spin}) gives 
  spin $(\frac{i_{\alpha}}{2})$-representation.
  For a complex torus $T$ of dimension $g$, 
   we have the following formula. 
   (See~\cite[Section~2]{HST}.)
  $$IH^{\ast}(T)=I_g \cneq \left[ \left(\frac{1}{2}
 \right) \oplus 2(0) \right] ^{\otimes g}.$$
\begin{lem}\label{choice}
The 
$\mathfrak{sl}_2$-representation type of $IH^{\ast}(X)$
does not depend on a choice of an ample divisor $H$.
\end{lem}
\begin{proof} 
Let $H'$ be another ample divisor, and 
consider the operator $\eta'(\ast)=H' \wedge \ast$. 
Let
$T_j, T_j' \subset IH^{\ast}(X)$
be the sub $\mathfrak{sl}_2$-representations
with respect to the operators $\eta, \eta'$ respectively, 
 consisting of direct sums of 
spin $(j')$-representations for $j' \ge j$. 
 Suppose that
 $T_j$, $T'_{j}$ 
have the same $\mathfrak{sl}_2$-represention types.
Then the minimal degrees of the following graded vector spaces, 
\begin{align}\label{minimal}IH^{\ast}(X)/T_j, \quad 
IH^{\ast}(X)/T_{j}',\end{align}
 are same, say $d\in (-2j, 0]$. 
   Also the 
   subspaces of degree $d$ elements in (\ref{minimal}) have 
   the same dimensions, say $l$. 
Then we see
$$T_{-\frac{d}{2}}\cong T_{j}\oplus \left( -\frac{d}{2} \right)^{\oplus l}, 
\quad T_{-\frac{d}{2}}'\cong T_{j}'\oplus
 \left( -\frac{d}{2} \right)^{\oplus l},$$
 as $\mathfrak{sl}_2$-representations. 
 Therefore $T_{-\frac{d}{2}}$ and $T_{-\frac{d}{2}}'$ have 
 the same $\mathfrak{sl}_2$-representation types. 
 By the induction we obtain the lemma. 
\end{proof}

   \subsection{Relative Lefschetz actions}\label{relative}
  Let $f\colon X\to A$ be a projective morphism between 
  normal projective varieties. 
  The idea of~\cite{HST} is to define the 
  $(\mathfrak{sl}_2)_{L}\times (\mathfrak{sl}_2)_{R}$-action 
  on $IH^{\ast}(X)$ using the Lefschetz operators in 
  fiber directions and base directions. 
   Let $H_A$, $H_{X/A}$ be an
  ample divisor on $A$, a relative ample divisor on $X$
  over $A$ respectively.
  We denote by $D(\mathbb{C}_X)$, $\mathrm{Perv}(\mathbb{C}_X)$
  the derived category of constructible sheaves on $X$
  (with its classical topology), 
  the heart of the middle perverse t-structure on $D(\mathbb{C}_X)$
  respectively. 
   We have the perverse Leray spectral 
  sequence, 
  $$E_2 ^{r,s}= H^r (A,  {^{p}R^{s}}f_{\ast}\iI C _X) 
  \Rightarrow IH^{r+s}(X), $$
  where $\iI C _X \in \mathrm{Perv}(\mathbb{C}_X)$
   is the intersection complex on $X$, and 
  ${^{p}R^{s}}f_{\ast}\iI C _X \in \mathrm{Perv}(\mathbb{C}_A)$
   is the $s$-th cohomology
   of $\dR f_{\ast}\iI C_X$ 
   with respect to the middle perverse
   t-structure on $D(\mathbb{C}_A)$.  
   It is known that the
    above spectral sequence degenerates at $E_2$-terms
    (cf.~\cite[Theorem 6.2.5]{BBD}), and we have 
   two operators, 
   $$
   \eta _L=H_{X/A} \wedge  \colon E_2 ^{r,s} \lr E_2 ^{r, s+2}, \quad
   \eta _R= H_A \wedge \colon E_2 ^{r,s} \lr E_2 ^{r+2,s},
   $$
   such that $\eta _L ^{s}\colon E_2^{r,-s}\stackrel{\cong}{\to}E_2^{r,s}$
   and $\eta _R ^{r} \colon E_2^{-r,s} \stackrel{\cong}{\to}E_2^{r,s}$. 
  As in Paragraph~\ref{rep}, these two actions define 
  an $(\mathfrak{sl}_2)_{L}\times (\mathfrak{sl}_2)_{R}$-action on
   $IH^{\ast}(X)$. (cf.~\cite[Corollary 2.1]{HST}.)
  Also by the same argument of Lemma~\ref{choice}, the 
  $(\mathfrak{sl}_2)_{L}\times (\mathfrak{sl}_2)_{R}$-representation type
  of $IH^{\ast}(X)$
  does not depend on $H_{X/A}$, $H_A$. 
  
  \subsection{Definition of HST (Hosono, Saito, Takahashi)
   invariants}\label{DGV}
  Let $X$ be a projective variety. 
  For $\beta \in N_1(X)$ and an ample 
  divisor $H$ on $X$, let $M^{\beta}$ be the 
  moduli space of $H$-Gieseker semistable sheaves $E$ on $X$
  (cf.~\cite{Hu}),
  pure of dimension one, 
  with numerical type $(\beta, 1)$. 
  Let $\widetilde{M}^{\beta}\to M^{\beta}$ be the 
  normalization. 
  By the same argument as in~\cite[Chapter 5, Section 4]{Mum}, 
  there is a natural map, 
  \begin{align*}
  \pi _{\beta}\colon M^{\beta}\ni E \longmapsto 
  s(E)
   \in \Chow_{\beta} (X),\end{align*}
   and
   let $S^{\beta}$ the normalization of the image of $\pi _{\beta}$. 
   The induced morphism $\pi _{\beta}\colon \widetilde{M}^{\beta} 
   \to S^{\beta}$
   is projective, hence defines an $(\mathfrak{sl}_2)_{L}\times (\mathfrak{sl}_2)_{R}$-action 
   on $IH^{\ast}(\widetilde{M} ^{\beta})$. One can rearrange its action 
   in the following formula, (cf.~\cite[Theorem 2.4]{HST}) 
   $$IH^{\ast}(\widetilde{M} ^{\beta})= \bigoplus _{g\ge 0}I_g \otimes 
   R_g(\beta), $$
   where $R_g(\beta)$ is a virtual $(\mathfrak{sl}_2)_{R}$-representation. 
   \begin{defi}\emph{\bf{\cite[Definition 3.6]{HST}}}\label{gvd}
   We define $\tilde{n}_{g}^{\beta}$ to be 
   \begin{align}\label{ng}
   \tilde{n}_{g}^{\beta}= \sum_j (-1)^{2j}(2j+1)\cdot N_j
   \in \mathbb{Z},\end{align}
   after writing $R_g(\beta)=\sum _{j}N_j \cdot (j)_{R}$ as a 
   virtual representation. 
   \end{defi}
   \begin{rmk}\emph{
   In~\cite[Definition 3.6]{HST}, the invariant $\tilde{n}_g^{\beta}$ 
   is defined as $\mathrm{Tr}_{R_g(\beta)}(-1)^{h_R}$, 
   which coincides with the formula (\ref{ng}). } 
   \end{rmk}
   
   \begin{rmk}\emph{As pointed out in~\cite{PT}, 
   the invariants $\tilde{n}_g^{\beta}$ are unlikely to be
   BPS-invariants discussed in~\cite{GV}. 
   We need to involve virtual classes to 
   define appropriate BPS-counting which are deformation invariant. 
    }   \end{rmk}
   There is an alternative way of defining $\tilde{n}_{g}^{\beta}$ 
   pointed out by~\cite{SS}, and it is much 
   more useful for our purposes. 
   Let $V$ be the space of an 
   $(\mathfrak{sl}_2)_{L}\times (\mathfrak{sl}_2)_{R}$-representation. 
   Then the operator $h_L +h_R$ defines the grading 
   $V=\oplus V_n$ with $e_RV_n \subset V_{n+2}$. 
   We can decompose $V$ into the direct sum of 
   the vector subspaces spanned by 
   $$v, e_Rv, \cdots, e_R ^{l-1}v,$$
   where $v\in V_{\alpha}$ for some $\alpha$, $e_R ^l v=0$, 
   and there is no $v' \in V$ with $e_R v'=v$. 
   Such a subspace is called a \textit{Jordan cell of size $l$ and 
   minimal degree $\alpha$}. Let $\nu _{l}^{\alpha}\in \mathbb{Z}_{\ge 0}$ be 
   \begin{align}\label{nu}
   \nu _l ^{\alpha}= \sharp \{ \mbox{Jordan cells of size }l \mbox{ and
   minimal degree }\alpha \mbox{ in }V\}.\end{align}
   Note that $\nu_l^{\alpha}$ depends only on the
   $e_{R}$-action and the grading induced by $h_L+h_R$. 
   We have the following. 
   \begin{prop}\label{alt}\emph{\bf{\cite{SS}}}
   For $V=IH^{\ast}(\widetilde{M} ^{\beta})$, we have 
   $$\tilde{n}_{g}^{\beta}= 
   \sum _{\alpha +l \ge 1}(-1)^{\alpha +g}l\nu _{l}^{\alpha}
   \left\{ \left( \begin{array}{c}\alpha +l+g \\ 2g +1 \end{array} \right)
   -\left( \begin{array}{c}\alpha +l+g-2 \\ 2g +1 \end{array} \right)
   \right\}.$$
   \end{prop}
   
   \section{Stability conditions on triangulated categories}\label{Stab}
   In this section we briefly recall the notion of stability conditions
   on triangulated categories~\cite{Brs1}, fix some notation
   and proves some fundamental properties. 
   \subsection{Generalities}
   The notion of stability conditions on a triangulated category 
   $\dD$ was introduced by T.~Bridgeland~\cite{Brs1} motivated by 
   M.~Douglas's work on $\Pi$-stability~\cite{Dou1}, \cite{Dou2}. 
   Here we only introduce its definition and some terminology used 
   in this paper, and do not explain its review
    too much. For the readers who are not familiar with~\cite{Brs1}, 
   we recommend 
   consulting the original paper~\cite{Brs1}. 
   \begin{defi}\label{stde}\emph{
   A stability condition on $\dD$ consists of data 
   $\sigma =(Z, \pP)$, 
   $$Z\colon K(\dD) \lr \mathbb{C}, \quad \pP (\phi)\subset \dD,$$
   where $Z$ is a group homomorphism, $\pP (\phi)$ is a subcategory 
   for each $\phi \in \mathbb{R}$ which satisfies,} 
   \begin{itemize}
 \item $\pP (\phi +1)=\pP (\phi)[1].$ 
\item  \emph{If $\phi _1 >\phi _2$ and $A_i \in \pP (\phi _i)$, then 
$\Hom (A_1, A_2)=0$.} 
\item \emph{ If $E\in \pP (\phi)$ is non-zero,
 then $Z(E)=m(E)\exp (i\pi \phi)$ for some 
$m(E)\in \mathbb{R}_{>0}$.} 
\item \emph{For a non-zero object $E\in \tT$, we have the 
following collection of triangles:
$$\xymatrix{
0=E_0 \ar[rr]  & &E_1 \ar[dl] \ar[rr] & & E_2 \ar[r]\ar[dl] & \cdots \ar[rr] & & E_n =E \ar[dl]\\
&  A_1 \ar[ul]^{[1]} & & A_2 \ar[ul]^{[1]}& & & A_n \ar[ul]^{[1]}&
}$$
such that $A_j \in \pP (\phi _j)$ with $\phi _1 > \phi _2 > \cdots >\phi _n$. }
\end{itemize}
\end{defi}
A non-zero object in $\pP(\phi)$ is called \textit{semistable of phase}
$\phi$. 
The mass of $E$ is defined to be 
$$m_{\sigma}(E)=\sum _{j}|Z(A_j)|.$$
   For an interval $I\subset \mathbb{R}$, we denote by $\pP (I)$ the 
   smallest extension closed subcategory of $\dD$ which 
   contains $\pP(\phi)$ with $\phi \in I$. 
   It is easy to see that $\pP ((0,1])$ is the heart of a
   bounded t-structure 
   on $\dD$. This gives an 
   alternative way of constructing a stability condition. 
   \begin{prop}\emph{\bf{\cite[Proposition 4.2]{Brs1}}}\label{equiv} 
   Giving a stability condition is equivalent to 
   giving the heart of a bounded t-structure $\aA \subset \dD$ together 
   with a group homomorphism $Z\colon K(\aA) \to \mathbb{C}$ such that 
   for a non-zero object $E\in \aA$ one has
\begin{align}\label{up} Z(E)\in \mathbb{H}\cneq 
\{ r\exp (i\pi \phi) \mid r>0, 0<\phi \le 1\},\end{align}
and the pair $(Z, \aA)$ 
satisfies the Harder-Narasimhan 
property. \end{prop}
The set of stability conditions which satisfy the 
local finiteness (cf.~\cite[Definition 5.7]{Brs1})
is denoted by $\Stab (\dD)$. 
It is shown in~\cite[Section 6]{Brs1} that $\Stab (\dD)$ has a 
natural topology. Furthermore 
for each connected component $\Sigma \subset \Stab (\dD)$, 
there exists a linear subspace $V(\Sigma )\subset \Hom _{\mathbb{Z}}(K(\dD), 
\mathbb{C})$ with a norm such that we have the local homeomorphism, 
(cf.~\cite[Theorem 1.2]{Brs1}) 
\begin{align*}
\zZ \colon \Sigma \ni (Z, \pP) \longmapsto Z\in V(\Sigma).\end{align*}

   \subsection{Stability conditions on Calabi-Yau 3-folds}
  For a Calabi-Yau 3-fold $X$, we consider 
   the following triangulated category as in the introduction, 
   $$\dD _X \cneq \{ E\in D(X) \mid \dim \Supp (E) \le 1\}.$$
   Here we introduce the subspace of $\Stab (\dD _X)$, coming from 
   the points corresponding to the neighborhood of 
   the large volume limit at $X$. 
   For $B +i\omega \in N^1 (X)_{\mathbb{C}}$, we set 
   $Z_{(B, \omega)}\colon K(\dD _X) \to \mathbb{C}$ as 
   $$Z_{(B, \omega)}(E)= -\ch _3 (E) +(B +i\omega)\cdot \ch _2 (E).$$
   \begin{rmk}\emph{
   Note that $Z_{(B,\omega)}$ factors as follows, 
   $$Z_{(B,\omega)}\colon K(\dD_X) \stackrel{(\ch_2,\ch_3)}{\lr}
    N_1(X)\oplus \mathbb{Z}
   \lr \mathbb{C}.$$
   Here the right arrow takes $v=(\beta,k)$ to $-k+(B+i\omega)\beta$, 
   which we write as $Z_{(B,\omega)}(v)$ by abuse of notation. }
   \end{rmk}

   We also set $\Coh _{\le 1}(X)$ as
   $\Coh _{\le 1}(X)\cneq \Coh (X) \cap \dD _X$. 
   Note that $\Coh _{\le 1}(X)$ is the heart of a bounded t-structure on $\dD _X$. 
   We have the following lemma. 
   \begin{lem}\label{LV}
   For $B +i\omega \in A(X)_{\mathbb{C}}$, the pair 
   $$\sigma _{(B, \omega)}\cneq (Z_{(B, \omega)}, \Coh _{\le 1}(X))$$
   determines a point in $\Stab (\dD _X)$
   in the sense of Proposition~\ref{equiv}.  
   \end{lem}
   \begin{proof}
   The proof of this lemma is exactly same as in~\cite[Lemma 4.1]{Tst}.
     In fact for a non-zero $E\in \Coh _{\le 1}(X)$, 
   we have $\Imm Z_{(B, \omega)}(E) >0$ when $\dim \Supp (E) =1$
   and $Z_{(B, \omega)}(E) \in \mathbb{R}_{<0}$ 
   when $\dim \Supp (E)=0$. 
   Thus (\ref{up}) holds. 
   One can also check the Harder-Narasimhan 
   property as in~\cite[Lemma 4.1]{Tst}. 
   \end{proof}
   
   \begin{rmk}\label{Gs}\emph{
   For an object $E\in \Coh_{\le 1}(X)$, it is 
   $\sigma_{(0, \omega)}$-semistable if and only if 
   for any non-trivial 
   subobject $F\subset E$ in $\Coh_{\le 1}(X)$, one has 
  $$\frac{\ch_3(F)}{\omega \cdot \ch_2(F)} \le 
  \frac{\ch_3(E)}{\omega \cdot \ch_2(E)}, $$
  i.e. $E$ is $\omega$-Gieseker semistable sheaf.}   
   \end{rmk}
   We define $\Stab(X)$ to be the following fiber product, 
   $$\xymatrix{
   \Stab(X) \ar[r]\ar[d] & \Stab (\dD _X) \ar[d] \\
   N^1 (X)_{\mathbb{C}} \ar[r] & \Hom (K(\dD _X), \mathbb{C}),
   }$$
   where the right arrow takes $(Z,\pP)$ to $Z$ and
   the bottom arrow takes $B +i\omega$ to $Z_{(B, \omega)}$. 
   Note that the stability conditions constructed in Lemma~\ref{LV} 
   are contained in $\Stab (X)$. Let  
   $U_X \subset \Stab (X)$ be
   $$U_X \cneq \{ \sigma _{(B, \omega)}\in \Stab(X) \mid B +i\omega \in 
   A(X)_{\mathbb{C}}\}.$$ 
  We have to check the following, whose proof will be postponed
  in Section~\ref{tech}. 
  \begin{lem}\label{op}
   The subset $U_X$ is an open connected subset of $\Stab (X)$. 
   \end{lem}
    By~\cite[Theorem 1.1]{Brs1} and Lemma~\ref{op},  the 
  map    
   $$\zZ_X \colon \Stab(X) \ni (Z,\pP) \longmapsto 
   Z \in  N^1 (X)_{\mathbb{C}},$$
   restricts to a homeomorphism between 
   $U_X$ and $A(X)_{\mathbb{C}}$. 
   
   \subsection{Wall and chamber structures}
   In this paragraph, we recall the notion of 
   wall and chamber structures on the space of 
   stability conditions. 
   The wall and chamber
    structure is introduced in~\cite{Brs2} 
    on the space of stability conditions on K3 surfaces, 
     and we show that 
   our space $\Stab(X)$ also possesses such a structure.  
   Let 
   $$ \Stab^{\ast}(X) \subset \Stab(X), $$
   be the connected component which contains $U_X$. 
   Let $\sS \subset \dD_X$ be a set of objects. 
   For $\sigma \in \Stab^{\ast}(X)$, 
   we call $\sS$ has \textit{bounded mass} if 
   there is $m>0$ such that 
   $$m_{\sigma}(E) <m,$$
   for any $E\in \sS$. It is easy to show that if the above
   condition holds for $\sigma \in \Stab^{\ast}(X)$, 
   then it also holds for any $\sigma ' \in \Stab^{\ast}(X)$. 
   We have the following proposition. 
   (cf.~\cite[Proposition 9.3]{Brs2}.)
   \begin{prop}\label{wallcham}
For a fixed compact subset $\overline{\uU}
\subset \Stab^{\ast}(X)$, there is a finite 
number of real codimension one submanifolds 
$\{ \wW_{\gamma} \mid \gamma \in \Gamma \}$ 
such that each connected component  
$$\cC \subset \overline{\uU}
\setminus \bigcup _{\gamma \in \Gamma}\wW_{\gamma}$$
has the following property. If $E\in \sS$ is $\sigma$-semistable for 
some $\sigma \in \cC$,
 then $E$ is semistable in $\sigma'$ for all $\sigma'
\in \cC$. 
 \end{prop}
   \begin{proof}
   The same proof of~\cite[Proposition 9.3]{Brs2}
   applies once we show the analogue of~\cite[Lemma 9.2]{Brs2}
   in our case, i.e. the set of numerical classes, 
   \begin{align}\label{once}
   \{ (\ch_2(E), \ch_3(E)) \mid E\in \sS \} \subset N_1(X) \oplus 
   \mathbb{Z},
   \end{align}
   is a finite set. 
   For an ample divisor $\omega$, let us take 
   $\sigma=\sigma_{(0, \omega)}$. 
   Also for $E\in \sS$, let $A_1, \cdots, A_n$ be the 
   $\sigma$-semistable factors. 
   Then the bounded mass condition for $\sS$ implies that the values 
  $$\lvert \ch_3(A_i) \rvert, \quad \lvert \ch_2(A_i)\cdot \omega \rvert, $$
  are bounded. On the other hand, the following space, 
  $$\{ c\in \overline{NE}(X) \mid c\cdot \omega \le 1\}
  \subset N_1(X), $$
  is compact. In fact openness of the ample cone immediately 
  implies the compactness of the above space. 
  Since $\ch_2(A_i)$ or 
  $-\ch_2(A_i)$ is contained in $\overline{NE}(X)$, 
  we see that the pair $(\ch_2(A_i), \ch_3(A_i))$ 
  has finite number of possibilities. Hence 
  (\ref{once}) is also finite. 
  \end{proof}
 We call a connected component 
 $\cC \subset \overline{\uU}
\setminus \bigcup _{\gamma \in \Gamma}\wW_{\gamma}$
\textit{chamber}. 
   
\subsection{Moduli theory of semistable objects}
In this paragraph, we consider the moduli problem of $\sigma$-semistable 
objects $E\in \dD_X$ for $\sigma \in \Stab(X)$. 
Such a moduli theory is studied in~\cite{Tst3} for 
stability conditions on K3 surfaces and abelian surfaces, 
and we use some basic arguments developed there. 
Let $\mM$ be the moduli stack of objects $E\in \dD_X$, satisfying 
the following condition, 
\begin{align}\label{Ext}
\Ext^i(E, E)=0, \quad \mbox{ for } i<0.\end{align}
Then by the result of Lieblich~\cite{LIE}, 
the stack $\mM$ is an Artin stack of 
locally finite type over $\mathbb{C}$. 
More precisely, Lieblich showed that the stack 
of $E\in D(X)$ satisfying the above condition is an Artin 
stack of locally finite type. However for a family 
of objects $\eE \in D(X\times S)$, the condition $\eE_s \in \dD_X$
is obviously an open condition, thus $\mM$ is also an Artin stack 
of locally finite type. 

 For $\sigma=(Z, \pP) \in \overline{U}_X$, 
$v\in N_1(X) \oplus \mathbb{Z}$ and $\phi \in \mathbb{R}$, we 
can consider the substack, 
$$i \colon \mM^{(v, \phi)}(\sigma) \hookrightarrow \mM, $$
which is the moduli stack of $E\in \pP(\phi)$ of numerical type $v$. 
The purpose here is to show that $\mM^{(v, \phi)}(\sigma)$ is algebraic. 
Recall that a set of objects $\sS \subset \dD_X$ is called 
\textit{bounded} if there is a finite type $\mathbb{C}$-scheme $S$
together with an object $\eE \in D(X\times S)$ such that 
any object $E\in \sS$ is isomorphic to $\eE_s$ for some $s\in S$. 
We use the following lemma, whose proof will be postponed until 
Section~6. 

\begin{lem}\label{bounded}
Let us take $\sigma =(Z, \pP)\in \Stab^{\ast}(X)$ such 
that $Z$ is defined over $\mathbb{Q}$. Then the 
set of objects, \emph{
$$\mM^{(v, \phi)}(\sigma)(\Spec \mathbb{C}) =
\{ E\in \pP(\phi) \mid 
E\mbox{ is of numerical type }v\}, $$}
is bounded.
\end{lem}
Using the above lemma and the argument in~\cite{Tst3}, 
we show the following proposition. 
\begin{prop}\label{Artin}
The stack $\mM^{(v, \phi)}(\sigma)$ is an Artin stack of finite 
type over $\mathbb{C}$, and $i\colon \mM^{(v, \phi)}(\sigma)
   \hookrightarrow \mM$ is an open immersion. 
   \end{prop}
   \begin{proof}
   For $B$ and $\omega$ are rational, the set of 
   $\mathbb{C}$-valued points of $\mM^{(v, \phi)}(\sigma_{(B, \omega)})$
   is bounded by Lemma~\ref{bounded}. 
   Hence by~\cite[Lemma 3.13 (ii), Proposition 3.18]{Tst3}
   the result is true for such $\sigma_{(B,\omega)}$. 
   Then one can apply~\cite[Theorem 3.20, Step 1]{Tst3}, 
   and conclude the result for any $\sigma \in \overline{U}_X$.  
\end{proof}

   \section{Motivic Gopakumar-Vafa invariants}\label{Mot}
   The purpose of this section is to introduce the invariants
    $n_g^{\beta}(X)\in \mathbb{Z}$, from a certain motivic 
    invariant of varieties over a Chow variety. 
   \subsection{Motivic invariants of varieties}
  Let $A$ be a projective variety over $\mathbb{C}$. 
  First let us recall the Grothendieck group of varieties over $A$. 
  \begin{defi}\emph{
  We define the $\mathbb{Z}$-module $K_0 (\Var /A)$ to be
  $$K_0 (\Var /A)= \bigoplus \mathbb{Z}[(X, \pi)]/\sim, $$
  where $[(X,\pi)]$ is an isomorphism class of
   a pair of a quasi-projective variety $X$ together with a 
   morphism $\pi \colon X\to A$.  
  The equivalence relation $\sim$ is generated by 
  $$[(X, \pi)]=[(Z, \pi|_{Z})]+[(X\setminus Z, \pi |_{X\setminus Z})],$$
  for closed subvarieties $Z\subset X$.} 
  \end{defi}
  Let $\pi \colon X\to A$ be a projective 
  morphism with $X$ smooth and connected.
  There is an induced morphism $X \to \widetilde{A}$, where
  $\widetilde{A}$ is the normalization of $\pi (A)$. 
  Then as in Paragraph~\ref{relative}, $IH^{\ast}(X)$ 
  carries an
  $(\mathfrak{sl}_2)_{L}\times (\mathfrak{sl}_2)_{R}$-action 
 with respect to the morphism 
 $X\to \widetilde{A}$. Let $\nu _l^{\alpha} \in \mathbb{Z}$ 
  be the number of Jordan cells in $IH^{\ast}(X)$
  for this action, 
  defined in (\ref{nu}). 
  We set $P(X,\pi)\in \mathbb{Z}[t,s]$ as follows, 
 $$P(X, \pi)= t^{\dim X}\sum _{\alpha, l} 
 \nu _{l}^{\alpha}t^{\alpha}s^{l-1}.$$
  We show the following proposition. 
   \begin{prop}
   There exists a map, 
   $$\Upsilon _A \colon K_0(\Var /A) \lr \mathbb{Z}[t, s],$$
   such that for any projective morphism $\pi \colon X \to A$ with 
   $X$ smooth and connected, we have 
   $$\Upsilon _A ([(X, \pi)])= P(X, \pi).$$
   \end{prop}
   \begin{proof}
    Let $X$ be a connected smooth projective 
   variety with a morphism $\pi \colon X\to A$.  
   Let $Z\subset X$ be a smooth closed subvariety and 
   take the blow-up $p\colon X^{\dag} \to X$ along $Z$, 
   \begin{align}\label{diag}
   \xymatrix{
   E \ar[r]^{j} \ar[d]_{q} & X^{\dag}\ar[d]_{p} \ar[dr]^{\pi^{\dag}} &\\
   Z \ar[r]_{i} & X \ar[r]_{\pi} & A.
   }\end{align}
   By the result of~\cite[Theorem 5.1]{Bit}, 
   $K_0(\Var/A)$ is generated by such $[(X,\pi)]$ with relation 
    $$[(X^{\dag},\pi^{\dag})]-[(E, \pi^{\dag}\circ j)]=
     [(X, \pi)]-[(Z,\pi\circ i)].$$
  Hence it is enough 
   to show that 
   \begin{align}\label{en}
   P(X^{\dag},\pi^{\dag})-P(E, \pi^{\dag}\circ j)= P(X, \pi)-P(Z,\pi\circ i).\end{align}
   Let $d=\dim X$ and $r$ be the codimension of $Z$ in $X$. 
   Note that $q\colon E\to Z$ is a $\mathbb{P}^{r-1}$-bundle. 
   By the isomorphism (\ref{iso1}) in Lemma~\ref{blow} below,
we have 
   \begin{align*}
  \frac{P(X^{\dag},\pi^{\dag})}{t^d}=\frac{P(X, \pi)}{t^d}+\frac{t^{-r+2}P(Z,\pi\circ i)}{t^{d-r}}
  +\cdots + \frac{t^{r-2}P(Z,\pi\circ i)}{t^{d-r}},
  \end{align*}
   since $\mu_{l}^{\alpha}$
 is determined by the grading and $e_R$-actions.  
 Thus we have 
  \begin{align*}
  P(X^{\dag}, \pi^{\dag})=P(X,\pi)+t^{2}P(Z,\pi\circ i)+ \cdots +t^{2(r-1)} P(Z,\pi\circ i).\end{align*}
  Similarly the isomorphism (\ref{iso2}) in Lemma~\ref{blow}
  yields,  
  $$P(E,\pi^{\dag}\circ j)= P(Z,\pi\circ i)+t^2 P(Z,\pi\circ i)+ \cdots + t^{2(r-1)}P(Z,\pi\circ i).$$
   Hence the equation (\ref{en}) holds. 
      \end{proof}
      
  We have used the following lemma. 
  \begin{lem}\label{blow}
  In the diagram (\ref{diag}), we have the
  following isomorphisms, 
  \begin{align}
  \label{iso1}
   IH^{\ast}(X^{\dag}) & \cong IH^{\ast}(X) \oplus 
   \bigoplus _{i=0}^{r-2}IH^{\ast}(Z)[-r+2+2i], \\
    \label{iso2}
   IH^{\ast}(E)&\cong \bigoplus_{i=0}^{r-1}
   IH^{\ast}(Z)[-r+1+2i], 
  \end{align}
  which preserve $e_R$-actions. 
  \end{lem}    
  \begin{proof}
  Let $\oO_E(1)$ be the tautological line bundle 
  on $E$ and set $\xi=c_1(\oO_E(1)) \in H^2(E, \mathbb{C})$. 
  We have the following isomorphism, 
 \begin{align}\label{iso3}
 H^{\ast}(Z, \mathbb{C})\oplus 
  H^{\ast}(Z, \mathbb{C})[-2] \oplus 
  \cdots \oplus H^{\ast}(Z, \mathbb{C})[-2(r-1)] 
  \stackrel{\cong}{\lr} H^{\ast}(E, \mathbb{C}), \end{align}
  which sends $(v_0, v_1, \cdots, v_{r-1})$ to 
  $\sum \xi^i q^{\ast}v_i$. 
  It is obvious that (\ref{iso3}) preserves $e_R$-actions, 
  hence (\ref{iso2}) follows. 
  
  Next in~\cite[p605]{GH}, it is shown that we have the isomorphism, 
  \begin{align}\label{shown}
  H^{\ast}(X^{\dag}, \mathbb{C}) \stackrel{\cong}{\lr}
  H^{\ast}(X, \mathbb{C}) \oplus 
  \left( H^{\ast}(E, \mathbb{C})/q^{\ast}H^{\ast}(Z, \mathbb{C}) \right),
  \end{align}
  which sends $v\in H^{\ast}(X, \mathbb{C})$ to 
  $(p_{\ast}v, j^{\ast}v)$. For an ample divisor $H_A$ on $A$, 
  we have 
  $$\pi^{\ast}H_A \wedge p_{\ast}v=p_{\ast}(p^{\ast}\pi^{\ast}H_A \wedge v), 
  \quad j^{\ast}\pi^{\dag \ast}H_A \wedge j^{\ast}v=
  j^{\ast}(\pi^{\dag \ast}H_A \wedge v), $$
  Hence the isomorphism (\ref{shown}) preserves $e_R$-actions. 
  Using (\ref{shown}), (\ref{iso2}), 
  we obtain the isomorphism (\ref{iso1}).   
  \end{proof}
 In the following, we compute some examples of our 
 invariant $P(X, \pi)$. 
 \begin{exam}\emph{
 (i) Let $A$ be a 
 $d$-dimensional smooth projective variety and consider the element
 $[(A, \id)] \in K_0(\Var /A)$.  
 Let $b_i(A) \in \mathbb{Z}$ be the $i$-th betti number of $A$.
 Then one can easily see  
$$\nu_{l}^{\alpha}= \left\{ \begin{array}{cl}
b_{d+\alpha}(A)-b_{d+\alpha -2}(A), & \mbox{ if }l=-\alpha +1, \\
0, & \mbox{ otherwise. }
\end{array}\right. $$ 
Hence we have 
$$P(A, \id)=
\sum_{\alpha=0}^{d}(b_{\alpha}(A)-b_{\alpha-2}(A))t^{\alpha}s^{-\alpha+d}.$$
(ii) Let $\pi \colon X \to A$ be a projective bundle with fiber 
$\mathbb{P}^{r-1}$. Then applying (i) and (\ref{iso2}), 
we have 
\begin{align*}
P(X, \pi)&=P(A, \id)(1+ t^2 +\cdots +t^{2(r-1)}), \\
&= \sum_{\alpha=0}^{d}\sum_{k=0}^{r-1}
(b_{\alpha}(A)-b_{\alpha-2}(A))t^{\alpha+2k}s^{-\alpha+d}.
\end{align*}
(iii) Let $i\colon Z\hookrightarrow A$
 be a smooth subvariety of codimension $r$. 
Let $\pi \colon X \to A$ be a blow-up along $Z$. Using (\ref{iso1}), we have
\begin{align*}
P(X, \pi)&= P(A, \id)+ P(Z, i)(t^2+t^4+ \cdots + t^{2(r-1)}), \\
&=\sum_{\alpha=0}^{d}(b_{\alpha}(A)-b_{\alpha-2}(A))t^{\alpha}s^{-\alpha+d}
+ \sum_{\alpha=0}^{d-r}\sum_{k=1}^{r-1}
(b_{\alpha}(Z)-b_{\alpha-2}(Z))t^{\alpha+2k}s^{-\alpha+d-r}.\end{align*}
 Here we have used Lemma~\ref{finite} below, which 
 shows $P(Z, i)=P(Z, \id)$. 
 Lemma~\ref{finite} will also be used in Lemma~\ref{comm}}
 \end{exam}
 \begin{lem}\label{finite}
 Let
 $u\colon A \to A'$ be a finite morphism between 
 projective varieties. Then the following diagram commutes, 
 $$\xymatrix{
   K_0(\Var/A) \ar[r]^{\Upsilon_{A}}\ar[d]_{u_{\ast}} & \mathbb{Z}[t,s] \\
   K_0(\Var/A')
   \ar[ru]_{\Upsilon_{A'}}. & }$$
   Here $u_{\ast}$ takes $\pi \colon X\to A$ to 
   $u\circ \pi \colon X\to A \to A'$. 
 \end{lem}
  \begin{proof}
  It is enough to check the commutativity for 
  $[(X, \pi)] \in K_0(\Var/A)$, where
  $X$ is a smooth projective variety and 
   $\pi\colon X\to A$ is a morphism. 
  Let $H_{A'}$ be an ample divisor on $A'$, and set 
  $H_A=u^{\ast}H_{A'}$. Note that the 
  divisor $H_A$ is also ample because $u$ is finite. 
  Since the definition of $\Upsilon_A$ does not depend on 
  a choice of an ample divisor on $A$, one may compute 
  $\Upsilon_A(X, \pi)$ using the $e_R$-action 
  on $IH^{\ast}(X)$ given by 
  $\pi^{\ast}H_A \wedge \ast$.  
  Also one may compute $\Upsilon_{A'}(X, u\circ \pi)$
  by the $e_R$-action 
  given by $(u\circ \pi)^{\ast}H_{A'}\wedge \ast$. 
  Since $H_A=u^{\ast}H_{A'}$, both $e_R$-actions are same, 
  hence $\Upsilon_A(X, \pi)=\Upsilon_{A'}(X, u\circ \pi)$. 
  \end{proof}
      \begin{rmk}\label{high}\emph{
  For $Q =\sum q^{a,b}t^a s^b\in \mathbb{Z}[t,s]$, let $m(Q)$ be the integer 
  $$m(Q) = \max \{a+2b \mid q^{a,b}\neq 0 \}.$$
    Let 
   $\pi \colon X\to A$ be a projective morphism 
   with $X$ smooth and connected.  
   Then it is clear that $m(Q)=2\dim X$
   for $Q=\Upsilon _A([(X,\pi)])$. 
   In particular one can recover 
   $\nu _{l}^{\alpha}$ from $t^{-m(Q)/2}Q$. 
   Using this and the motivic property of $\Upsilon _A$,
    one can easily show that 
   for any quasi-projective variety $X$
    with a morphism $\pi \colon X \to A$ and
  $Q=\Upsilon _A([(X,\pi)])$, the integer $m(Q)$ is even.   }
   \end{rmk}
   Based on Remark~\ref{high} we introduce the following notation,
   which will be used in Paragraph~\ref{CouIn}.  
   \begin{defi}\label{daggar}\emph{
   We define the subset $\mathbb{Z}[t,s]^{\dag} \subset \mathbb{Z}[t,s]$
   to be the set of polynomials $Q$ such that
   $m(Q)$ is even. 
   For $Q \in \mathbb{Z}[t,s]^{\dag}$, we define 
   the operation $Q^{\flat}$ by 
   $$Q^{\flat}\cneq t^{-m(Q)/2}Q \in \mathbb{Z}[t^{-1}, t, s].$$}
   \end{defi} 
   When $A=\Spec \mathbb{C}$, we write $K_0(\Var/A)$ as 
  $K_0(\Var/\mathbb{C})$, and just write its elements as
  $[X]\in K_0(\Var/\mathbb{C})$ omitting the structure morphism 
  $X\to \Spec \mathbb{C}$. 
     Also we write $\Upsilon_{\Spec \mathbb{C}}$ 
   as $\Upsilon$ for simplicity.  
   \begin{rmk}\emph{
   If $A=\Spec \mathbb{C}$, all the Jordan cells have length one. 
   Hence $\Upsilon(X) \in \mathbb{Z}[t]$
   is nothing but the virtual Poincar\'e polynomial of $X$.} 
   \end{rmk}
   There is a ring structure on $K_0(\Var /\mathbb{C})$ 
   defined by 
   $$([X_1], [X_2]) \longmapsto [X_1 \times X_2],$$
   with unit $[\Spec \mathbb{C}]$.    
   Also we have the following natural map, 
   $$\Pi \colon 
   K_0(\Var /\mathbb{C}) \times K_0(\Var /A) 
   \lr K_0(\Var /A),$$
   which takes the pair $([T], [(X,\pi)])$ to 
   $$T\times X \stackrel{\mathrm{pr}}{\lr} X \stackrel{\pi}{\lr} A,$$
   where $\mathrm{pr}$ is the projection to $X$. 
   The operation $\Pi$ makes $K_0(\Var /A)$ a $K_0(\Var /\mathbb{C})$-algebra. 
   We have the following lemma. 
   \begin{lem}\label{fol}
   The following diagram is commutative. 
   $$\xymatrix{
  \Pi\colon
   K_0(\Var /\mathbb{C}) \times K_0(\Var /A) \ar[d]_{\Upsilon
   \times \Upsilon _A} \ar[r]
   &K_0(\Var /A) \ar[d]_{\Upsilon _A} \\
   \mathbb{Z}[t] \times \mathbb{Z}[t,s] \ar[r] & \mathbb{Z}[t,s],
   }$$
   where the bottom arrow takes $(Q_1(t), Q_2(t,s))$ to 
   $Q_1(t)Q_2(t,s)$. 
      \end{lem}
  \begin{proof}
  It is enough to check
  \begin{align}\label{rin}
  \Upsilon _A([T\times X, \pi \circ \mathrm{pr}])
  =\Upsilon(T)\cdot
   \Upsilon _A ([(X,\pi)]),\end{align}
   for smooth projective $T$ and $X$
   with a morphism $\pi \colon X\to A$. 
  We have the isomorphism as graded vector spaces, 
  $$IH^{\ast}(T\times X) \cong IH^{\ast}(T) \otimes IH ^{\ast}(X).$$
  Let us introduce the $e_R$-action
   on the right hand side, 
  by letting $e_R$ act on $IH^{\ast}(T)$ trivially. 
 Then the above isomorphism preserves the $e_R$-actions,
 which shows the 
 equality (\ref{rin}) immediately. 
  \end{proof}
  
  \begin{rmk}\label{thank}\emph{
  More generally a motivic invariant is defined to be
  a ring homomorphism $\Upsilon \colon K(\Var/ \mathbb{C}) \to 
  \Lambda$ for some ring $\Lambda$. 
  Then a motivic invariant relative to $A$ is defined 
  to be a map, 
  $$K_0(\Var/A) \otimes _{K_0(\Var/ \mathbb{C})} \Lambda 
  \lr M, $$
  for some $\Lambda$-module $M$. Lemma~\ref{fol}
  implies that $\Upsilon_A$ is obtained in this way.}
  \end{rmk}
  
  \begin{rmk}\label{fib}\emph{
  Let $\pi \colon X\to A$ be a morphism and
 $p \colon Z\to X$ be a Zariski locally trivial fibration with fiber $F$. 
 Then we have 
 $[(Z, \pi\circ p)]=\Pi ([F], [(X,\pi)])$,
 which yields
 the equality $\Upsilon _A([(Z, \pi\circ p)])=\Upsilon(F)\cdot
  \Upsilon _{A}([(X,\pi)])$
  by Lemma~\ref{fol}.
  (cf.~\cite[Lemma 4.2]{Joy5}.)} \end{rmk}

  \subsection{Motivic invariants of Artin stacks}
  Here we extend the invariant constructed in the
  previous paragraph to the invariant of Artin stacks over a
  projective variety $A$. 
  The material of this paragraph is a slight generalization of 
  Joyce's work~\cite{Joy5}. In \textit{loc.cite,} he 
  works on the motivic invariants over $K_0(\Var/\mathbb{C})$
  such as virtual Poincar\'e polynomials. For our purpose we have to 
  extend the results in~\cite{Joy5} to invariants over 
  $K_0(\Var/A)$ such as $\Upsilon _A$. However the proofs are straightforward 
  generalizations and 
  we will leave some details to readers. 
  Let $\rR$ be an Artin stack of locally finite type over 
  $\mathbb{C}$.  
  Following Joyce~\cite{Joy5}, we introduce the Grothendieck 
  group of Artin stacks over $\rR$, 
  denoted by $K_0(\St/\rR)$ in this paper. 
   It was denoted by $\SF(\rR)$ 
  in Joyce's papers~\cite{Joy2}, \cite{Joy3}, \cite{Joy4}, \cite{Joy5}. 
  \begin{defi}\label{SFD}\emph{\bf{\cite[Definition 3.1]{Joy5}}}\emph{
  We define $K_0(\St/\rR)$ to be the $\mathbb{Q}$-vector space 
  generated by equivalence classes of 
  pairs $[(\xX, \rho)]$, where $\rho \colon \xX \to \rR$ is 
  a 1-morphism of Artin stacks, $\xX$ is of finite type over $\mathbb{C}$ 
  with affine geometric stabilizers, such that for each closed 
  substack $\yY \subset \xX$, one has }
  $$[(\xX, \rho)]=[(\yY, \rho |_{\yY})]+[(\xX \setminus \yY, \rho |_{\xX 
  \setminus \yY})].$$
  \end{defi}
  
  The following lemma is a generalization of~\cite[Theorem 4.10]{Joy5}.
  Below we set $\Lambda=\mathbb{Q}(t,s)$. 
  \begin{lem}\label{general}
  Let $A$ be a projective variety. Then 
  $\Upsilon _{A}\colon K_0(\Var /A) \to \mathbb{Z}[t,s]$ 
  extends to the map, 
  $$\Upsilon _{A}' \colon K_0(\St/A) \lr \Lambda=\mathbb{Q}(t,s),$$
  such that for a 1-morphism $\rho \colon \xX \to A$
  with $\xX \cong [X/G]$, $X$ is a quasi-projective variety 
  and $G$ a special algebraic 
  $\mathbb{C}$-group (cf.~\cite[Definition 2.1]{Joy5}), we have 
  \begin{align}\label{quot}
  \Upsilon _A '([(\xX, \rho)])=\frac{\Upsilon _A([(X, \pi)])}
  {\Upsilon([G])}.\end{align}
  Here $\pi$ is the composition, 
  $\pi \colon X \to [X/G] \cong \xX \stackrel{\rho}{\to} A$. 
  \end{lem}
  \begin{proof}
  When $A=\Spec \mathbb{C}$, Lemma is proved in~\cite[Theorem 4.10]{Joy5}. 
  Also 
  $\Upsilon([G])$
  is non-zero in $\mathbb{Z}[t]$ 
   by~\cite[Lemma 4.7]{Joy5},
   hence the RHS of (\ref{quot}) makes sense. 
   We have to check that
  for $\rho \colon \xX \to A$ with $\xX$ 1-isomorphic to 
  $[X/G]$, the value $\Upsilon _{A}([(X,\pi)])/\Upsilon([G])$
  does not depend on a choice of $X$, $G$, 
  and an isomorphism $\xX \cong [X/G]$.  This follows from 
  Remark~\ref{fib} and exactly the same 
  proof of~\cite[Proposition 4.8]{Joy5}.  
  Finally as in the proof of~\cite[Theorem 4.10]{Joy5}, any 
  Artin stack with affine geometric stabilizers $\xX$ 
  is stratified by global quotient stacks, and can 
  define $\Upsilon _{A}'([(\xX, \rho)])$ 
  by the formula (\ref{quot}) and the linearity. 
  Then the same proof of~\cite[Theorem 4.10]{Joy5} shows that
  $\Upsilon _{A}'([(\xX, \rho)])$
  does not depend on a choice of such a stratification.
  \end{proof}
   We will use a $\Lambda$-module with more relations than $K_0(\St/\rR)$, 
   denoted by $K_0(\St/\rR)$. 
   It was denoted by $\SF(\rR, \Upsilon, \Lambda)$ in~\cite{Joy5}. 
   \begin{defi}\emph{\bf{\cite[Definition 4.11]{Joy5}}}\label{relate}
   \emph{We define 
   $K_0(\St/\rR)_{\Upsilon}$ 
   to be the $\mathbb{\Lambda}$-module
  generated by equivalence classes of 
  pairs $[(\xX, \rho)]$, where $\rho \colon \xX \to \rR$ is 
  a 1-morphism of Artin stacks,
   $\xX$ is of finite type over $\mathbb{C}$ 
  with affine geometric stabilizers, such that}
  
  \emph{(i) For each closed 
  substack $\yY \subset \xX$, one has 
  $$[(\xX, \rho)]=[(\yY, \rho |_{\yY})]+[(\xX \setminus \yY, \rho |_{\xX 
  \setminus \yY})].$$
   \quad (ii) Let $\xX$ be a finite type Artin $\mathbb{C}$-stack with 
  affine geometric stabilizers together with a 1-morphism 
  $\rho\colon \xX \to \rR$. Let
  $T$ be a quasi-projective variety, 
   and $\mathrm{pr} \colon T\times \xX \to \xX$ the projection. Then 
   $$[(T\times \xX, \rho \circ \mathrm{pr} )]=
   \Upsilon([T])[(\xX, \rho)].$$
  \quad (iii) Let $\rho \colon \xX \to \rR$ be as above and $\xX \cong [X/G]$
   with $X$ quasi-projective, $G$ a special algebraic group 
   acting on $X$. Then we have
   $$[(\xX, \rho)]=\Upsilon([G])^{-1}[(X, \pi)],$$
   where $\pi$ is the composition $\pi \colon 
   X \to [X/G] \cong \xX \stackrel{\rho}{\to}\rR$. 
   } 
   \end{defi}

   \begin{rmk}\label{gene}\emph{
   As in the proof of~\cite[Theorem 4.10]{Joy5},
   any Artin stack of finite type is stratified by 
   global quotient stacks. Then
    using (i), (iii) in Definition~\ref{relate}, 
   one can show that
    $\Lambda$-module $K_0(\St/\rR)$
   is spanned over $\Lambda$ 
   by $[(X, \pi)]$, where $X$ is a variety and 
   $\pi \colon X\to \rR$ is a 1-morphism. 
   }
   \end{rmk}
  
   Now we descend the map $\Upsilon _A '$ to 
   $K_0(\St/A)_{\Upsilon}$. 
   \begin{lem}\label{descend}
   There is a $\Lambda$-module 
   homomorphism $\bar{\Upsilon} _A \colon 
   K_0(\St/A)_{\Upsilon} \to \Lambda$
   such that the following diagram commutes, 
   $$\xymatrix{
   K_0(\St/A) \ar[r]^{\Upsilon_{A}'}\ar[d] & \Lambda \\
   K_0(\St/A)_{\Upsilon}
   \ar[ru]_{\bar{\Upsilon} _A}. & }$$
   Here the left arrow is the natural quotient map. 
   \end{lem}
   \begin{proof}
   For the relation (ii) of Definition~\ref{relate}, we have 
   $$\Upsilon _A '([(T\times \xX , \rho \circ \mathrm{pr})])=
   \Upsilon(T)\cdot\Upsilon _A '([\xX, \rho]),$$
   by Lemma~\ref{fol} and the construction of $\Upsilon _A'$
   in Lemma~\ref{general}. 
   The compatibility with Definition~\ref{relate} (iii) 
   follows from (\ref{quot}). 
   \end{proof}
   Let $p \colon \rR ' \to \rR$ be a 1-morphism 
   between Artin stacks of locally finite type. 
    Then there is the notion of its 
   push-forward, 
   \begin{align}\label{forward}
   p _{\ast}\colon K_0(\St/\rR') \lr K_0(\St/\rR),\end{align}
   by taking a 1-morphism $\rho \colon \xX \to \rR'$ 
   to $p \circ \rho \colon \xX \to \rR$. 
   Moreover if $p$ is of finite type, there is the
   notion of its pull-back, 
   \begin{align}\label{back}
   p ^{\ast}\colon K_0(\St/\rR) \lr K_0(\St/\rR'),\end{align}
   by taking a 1-morphism $\rho \colon \xX \to \rR$
   to the fiber product $\xX \times _{\rR}\rR ' \to \rR '$. 
   (cf.~\cite[Definition 3.4]{Joy5}.)
   These operations descend to $\Lambda$-module 
   homomorphisms between $K_0(\St/\rR)_{\Upsilon}$
   and $K_0(\St/\rR')_{\Upsilon}$.
   (cf.~\cite[Theorem 4.13]{Joy5}.)

   \subsection{Ringel-Hall product}
   Let $\mM$ be the moduli stack of $E\in \dD_X$ satisfying 
   (\ref{Ext}), which we discussed in Section~3. 
   Let $\aA \subset \dD _X$ be the heart of a bounded t-structure, 
   and take $v\in N_1(X)\oplus \mathbb{Z}$. 
   We consider the substacks 
   $$\mathfrak{Obj}(\aA)\subset \mM, \quad \mathfrak{Obj}^v(\aA)\subset 
   \mM$$
   the stack of objects in 
   $\aA$, the stack of objects in $\aA$ of numerical type $v$ respectively. 
    Suppose that $\mathfrak{Obj}(\aA)$ is an open substack of $\mM$, 
  hence in particular $\mathfrak{Obj}(\aA)$, $\mathfrak{Obj}^v(\aA)$ are
   Artin stacks of locally finite type. 
   (This condition holds when $\aA =\Coh _{\le 1}(X)$.)
   Then there is an associative product 
   $\ast$ on $K_0(\St/\mathfrak{Obj}(\aA))$
    and $K_0(\St/\mathfrak{Obj}(\aA))_{\Upsilon}$, 
    based on \textit{Ringel-Hall 
   algebras}~\cite{Joy2}. 
   Let $\mathfrak{Ex}(\aA)$ the stack of the exact sequences, 
   $0\to E_1\to E_2 \to E_3 \to 0$ in $\aA$. We have the 
   diagram, 
   \begin{align}\label{Hall}
   \xymatrix{
   \mathfrak{Ex}(\aA) \ar[r]^{p_{2}}
    \ar[d]_{(p_1, p_3)} & \mathfrak{Obj}(\aA) \\
   \mathfrak{Obj}(\aA)\times \mathfrak{Obj}(\aA),}\end{align}
   where $p_{i}$ takes $0\to E_1 \to E_2 \to E_3 \to 0$ to 
   $E_i$. For elements $f_i \in K_0(\St/\mathfrak{Obj}(\aA))$
    with $i=1,2$, one can define $f_1 \ast f_2$ as 
    \begin{align}
    \label{ast}f_1 \ast f_2 =p_{2\ast}(p_1 \times p_3)^{\ast}(f_1 \times f_2)
    \in K_0(\St/\mathfrak{Obj}(\aA)),\end{align}
    where $p_{2\ast}$, $(p_1 \times p_3)^{\ast}$ are defined in 
    (\ref{forward}), 
    (\ref{back}) respectively. 
    Note that if $f_i \in K_0(\St/\mathfrak{Obj}^{v_i}(\aA))$, the element
    $f_1 \ast f_2$ is contained in $K_0(\St/\mathfrak{Obj}^{v_1 +v_2}(\aA))$. 
    See~\cite[Section 5]{Joy2} for the detail. 
    \begin{thm}\emph{\bf{\cite[Theorem 5.2]{Joy2}}}
    The operation $\ast$ makes $K_0(\St/\mathfrak{Obj}(\aA))$ and 
    $K_0(\St/\mathfrak{Obj}(\aA))_{\Upsilon}$ associative 
   algebras with unit $[(\Spec \mathbb{C}, \rho _0)]$, 
   where $\rho _0 \colon \Spec \mathbb{C} \to \mathfrak{Obj}(\aA)$
   corresponds to the zero object. 
    \end{thm}
    In order to simplify the expositions in 
   the following sections, 
  we introduce 
    \textit{Hall algebras of derived categories} 
   $(\hH (X), \ast)$, introduced by To$\ddot{\textrm{e}}$n~\cite{Toen}.
   We emphasize that in our proof, we will not use the 
   algebra $(\hH(X), \ast)$ essentially. 
   The algebra $(\hH(X), \ast)$ 
   contains $(K_0(\St/\mathfrak{Obj}(\aA)), \ast)$ as a subalgebra, 
   and all the computations in the proof will be made in the 
   latter algebra for suitable hearts of bounded t-structures 
   $\aA \subset \dD_X$. 
     As we do not need its actual definition, 
   we only give its rough explanation and properties. 
   Let
  $\widetilde{\dD}$ be a dg-category 
  of finite type (cf.~\cite[Definition 2.4]{Toen2}) 
   whose 
  homotopy category is $D(X)$. 
  Then To$\ddot{\textrm{e}}$n and Vaqui{\'e}~\cite[Theorem 0.1]{Toen2} 
  showed that the 
  stack of objects in $\widetilde{\dD}$, which they
  denote $\underline{Perf}(X)$ in~\cite[Definition 3.28]{Toen2},  
  is an $\infty$-stack of 
  locally geometric and locally of finite presentation.
  Then $\hH (X)$ is defined by
   the $\infty$-stack 
   version of our notion of stack functions over 
   $\underline{Perf}(X)$,
   (in~\cite[Paragraph 3.3]{Toen},
    it is denoted by $\hH_{abs}(\widetilde{\dD})$,)
   and the $\ast$-product is defined in the similar way of (\ref{ast}). 
   Furthermore by~\cite[Corollary 3.21]{Toen2},
   the stack $\mM$ is realized as an open substack 
   of $\underline{Perf}(X)$.  
  Thus if there is a 1-morphism,
   $$\rho \colon \xX \lr \mM,$$
   with $\xX$ an Artin stack of finite type, 
   it defines an element $[(\xX, \rho)]\in \hH (X)$. 
   In particular for the heart of a bounded t-structure 
   $\aA \subset \dD_X$ with $\mathfrak{Obj}(\aA)\subset \mM$
   open, the algebra $(K_0(\St/\mathfrak{Obj}(\aA)),\ast)$ 
   is realized as a subalgebra of $\hH(X)$. 
  (See~\cite{Toen}, \cite[Paragraph 3.3]{Toen3} for the detail.)

   \subsection{Counting invariants of moduli stacks}\label{CouIn}
   Below we assume $X$ is a Calabi-Yau 3-fold, and use the 
   space of stability conditions $\Stab(X)$ 
   and the open subset $U_X$ 
   in Section~\ref{Stab}. 
    Let us take $v\in N_1(X)\oplus \mathbb{Z}$,
    $\sigma \in \overline{U}_X$ and $\phi \in \mathbb{R}$. 
   We consider the substack 
   $$i\colon \mM ^{(v,\phi)}(\sigma)\hookrightarrow \mM,$$
   which is the moduli stack of $E\in \pP (\phi)$
  of numerical type $v$. 
  By Proposition~\ref{Artin}, $\mM^{(v, \phi)}(\sigma)$ is an Artin stack 
  of finite type. 
   Also for an interval $I\subset \mathbb{R}$, 
   let $C_{\sigma}(I) \subset N_1(X) \oplus \mathbb{Z}$ be the image 
   of the map, 
   $$(\ch _2, \ch _3) \colon \pP (I)\setminus \{0\}
     \lr N_1 (X) \oplus \mathbb{Z}.$$
  \begin{defi}
 \label{inv}\emph{
   For $\sigma \in \overline{U}_X$, we
    define $\delta ^{(v, \phi)}(\sigma)\in \hH(X)$
   to be 
   $$\delta^{(v,\phi)}(\sigma)=[(\mM ^{(v, \phi)}(\sigma),i)]
   \in \hH(X).$$
  Also
   define $\epsilon ^{(v, \phi)}(\sigma)\in 
   \hH(X)$ as follows, 
   \begin{align}\label{eps}\epsilon ^{(v, \phi)}(\sigma)
   =\sum _{v_1 + \cdots +v_n=v}\frac{(-1)^{n-1}}{n}
   \delta^{(v_1, \phi)}(\sigma)\ast \cdots \ast 
   \delta^{(v_n, \phi)}(\sigma),\end{align}
   where $v_i \in C_{\sigma}(\phi)$.} 
   \end{defi}  
  We have to check the following, whose proof will 
  be given in Section~\ref{tech}.   
 \begin{lem}\label{fin}
 The sum (\ref{eps}) is a finite sum.
 \end{lem}
  \begin{rmk}\label{alg}\emph{
  Suppose that for $\phi \in \mathbb{R}$, 
  there is the heart of a t-structure $\aA \subset \dD _X$
  such that $\mathfrak{Obj}(\aA)\subset \mM$ is open and 
  $\pP (\phi)\subset \aA$. Then $\delta ^{(v, \phi)}(\sigma)$
  is contained in the subalgebra
  $K_0(\St/\mathfrak{Obj}(\aA))$ and
  $\epsilon ^{(v, \phi)}(\sigma)$ coincides with the one 
  defined in the algebra $K_0(\St/\mathfrak{Obj}(\aA))$
  as in~\cite[Definition 8.1]{Joy3}.  
  } \end{rmk}
  
  \begin{rmk}\label{ed}\emph{
  Suppose any $\sigma$-semistable object of numerical type $v$ 
  is stable. Then we must have $\epsilon^{(v,\phi)}(\sigma)=\delta^{(v,\phi)}(\sigma)$. }  
  \end{rmk}
  
  For $v=(\beta, k)\in N_1(X)\oplus \mathbb{Z}$, let
   us consider the following open substacks of $\mM$, 
   $$\mathfrak{Coh}(X) \cneq \mathfrak{Obj}(\Coh _{\le 1}(X)), \quad 
   \mathfrak{Coh}^v(X) \cneq \mathfrak{Obj}^v(\Coh_{\le 1}(X)).$$
   By Remark~\ref{alg}, for $0< \phi \le 1$ and $\sigma \in U_X$, we have 
   \begin{align}\label{sf}
   \epsilon ^{(v, \phi)}(\sigma) \in K_0(\St/\mathfrak{Coh}^v(X)).
   \end{align}
   
   Using the same argument as in~\cite[Chapter 5, Section 4]{Mum},
   we have the following 1-morphism, 
   $$\pi \colon \mathfrak{Coh}^{v}(X)\ni E \longmapsto s(E) \in
   \Chow _{\beta}(X). $$
   Now we define the element $P(v, \sigma) \in \Lambda$. 
   \begin{defi}\label{mode}\emph{
   For $v\in N_1(X) \oplus \mathbb{Z}$ and $\sigma \in U_X$
   as above, we
   define $P(v, \sigma)\in \Lambda$ as follows. }
   \begin{itemize}
   \item \emph{If $v\in C_{\sigma}(\phi)$ with 
   $0<\phi \le 1$,
   we define 
   $$P(v, \sigma)\cneq (\mathbb{L}-1)\Upsilon _A '
   (\pi_{\ast}\epsilon ^{(v,\phi)}(\sigma)),$$
   where 
   $\mathbb{L}=\Upsilon(\mathbb{A}^1)$ and
   $A=\Chow _{\beta}(X)$.
   Note that $\pi_{\ast}\epsilon^{(v,\phi)}(\sigma)\in K_0(\St/A)$ makes sense 
   by (\ref{sf}). }
   \item \emph{If $v\in C_{\sigma}(\phi)$ with 
   $1<\phi \le 2$, we define 
   $$P(v, \sigma)\cneq P(-v, \sigma).$$
   Note that in this case $-v \in C_{\sigma}(\phi -1)$, hence 
   the RHS makes sense.}
   \item \emph{Otherwise we 
   define $P(v, \sigma)=0$. }
   \end{itemize} 
   \begin{rmk}\label{observe}
   \emph{The definition of $P(v, \sigma)$ 
   for $v\in C_{\sigma}(\phi)$ with $1<\phi \le 2$ is motivated 
   by the following observation. Let $\mathfrak{Coh}^v(X)[1]$ be 
   the stack of objects $E\in \Coh_{\le 1}(X)[1]$ of 
   numerical type $v$. Then we have 
   $$\epsilon^{(v, \phi)}(\sigma) \in K_0(\St/\mathfrak{Coh}^v(X)[1]).$$
   Also we have the 1-morphism, 
   $$\pi' \colon \mathfrak{Coh}^{v}(X)[1] \ni E \longmapsto 
   -s(E) \in A'\cneq \Chow_{-\beta}(X).$$
   Hence it is reasonable to define 
   \begin{align}\label{rea}
   P(v, \sigma)=(\mathbb{L}-1)\Upsilon_{A'}'(\pi'_{\ast}\epsilon^{(v, \phi)}
   (\sigma)).\end{align}
   However we can easily see that the RHS of (\ref{rea}) is equal to 
   $P(-v, \sigma)$, hence we define $P(v, \sigma)$ as in Definition~\ref{mode}
   to reduce the exposition. 
   }   
   \end{rmk}
   \end{defi}
    We have the following proposition. (Recall that
    we have defined $\mathbb{Z}[t,s]^{\dag}$
    and $Q^{\flat}$ for $Q\in \mathbb{Z}[t,s]^{\dag}$
    in Definition~\ref{daggar}.)
    \begin{prop}\label{dag}
    If $\sigma =\sigma _{(0, \omega)}$
    with $\omega$ ample
    and $v=(\beta, 1)$, we have 
    $P(v, \sigma)\in \mathbb{Z}[t,s]^{\dag}$.    
    \end{prop}
    \begin{proof}
    We may assume
    $v\in C_{\sigma}(\phi)$ for
     $0<\phi \le 1$. 
   If $\sigma =\sigma _{(0,\omega)}$
   and $v=(\beta, 1)$, any semistable object in $\sigma$ of 
   type $v$ is stable.
   In fact suppose there is a $\sigma$-semistable object 
   $E\in \Coh _{\le 1}(X)$ which is not stable. 
   Then there is an exact sequence  
   $0 \to E_1 \to E \to E_2 \to 0$ in $\Coh _{\le 1}(X)$
   such that
   \begin{align}\label{arg1}\arg Z_{(0,\omega)}(E)=\arg Z_{(0,\omega)}(E_1)
   =\arg Z_{(0,\omega)}(E_2).\end{align}
   Because $\ch _3(E)=1$, we have $\Ree Z_{(0, \omega)}(E)<0$, 
    hence (\ref{arg1}) implies 
    $\Ree Z_{(0, \omega)}(E_i)<0$
    for $i=1,2$.
    Since $-\ch _3(E_i)=\Ree Z_{(0, \omega)}(E_i)$, this
      contradicts to $1=\ch _3(E)=\ch _3(E_1)+\ch _3(E_2)$.
      
    By Remark~\ref{ed} we have  
   $$\epsilon^{(v,\phi)}(\sigma)=\delta^{(v,\phi)}(\sigma)=
   [([M^v /\mathbb{G}_m], \rho)],$$
   for some projective variety $M^v$, $\mathbb{G}_m$ acting on 
   $M^v$ trivially and 
   $\rho$ is a 1-morphism $[M^v /\mathbb{G}_m] \to \mathfrak{Coh}(X)$. 
    (The factor $\mathbb{G}_m$ comes from 
   the stabilizers $\Aut(E) \cong \mathbb{G}_m$ for stable objects
   $E$.)
   In fact $M^v$ is the moduli space of $\omega$-Gieseker stable 
   sheaves of Chern character $v$. 
    Thus by Remark~\ref{high}
    and noting $\mathbb{L}-1 =\Upsilon(\mathbb{G}_m)$, 
    we have $P(v,\sigma)\in \mathbb{Z}[t,s]^{\dag}$. 
    \end{proof}
    
     Now we define the notion of 
  motivic Gopakumar-Vafa invariant. 
  \begin{defi}\label{mgd}\emph{
  For $\sigma =\sigma _{(0, \omega)}\in U_X$ and $v=(\beta, 1)$, 
  write $P(v, \sigma)^{\flat}$ as 
  $$P(v, \sigma)^{\flat}=
  \sum _{\alpha, l}\nu_{l}^{\alpha}(\beta)t^{\alpha}s^{l-1}.$$
  By Proposition~\ref{dag}, it is possible to 
  define $P(v, \sigma)^{\flat}$. Then define the 
  \textit{motivic Gopakumar-Vafa invariant}
  $n_{g}^{\beta}(X)$ as follows, }
  \begin{align}\label{gvde}n_{g}^{\beta}(X)=
  \sum _{\alpha +l \ge 1}(-1)^{\alpha +g}l\nu _{l}^{\alpha}(\beta)
   \left\{ \left( \begin{array}{c}\alpha +l+g \\ 2g +1 \end{array} \right)
   -\left( \begin{array}{c}\alpha +l+g-2 \\ 2g +1 \end{array} \right)
   \right\}.\end{align}
   \end{defi}
   \begin{rmk}\emph{
   The motivation of (\ref{gvde}) in 
   Definition~\ref{mgd} comes from Proposition~\ref{alt}. 
  Obviously the invariant $n_g^{\beta}(X)$ coincides with 
  $\tilde{n}_g ^{\beta}$ if
  $\beta$ is represented by an effective one cycle, and
  the moduli space $M ^{\beta}$ in Paragraph~\ref{DGV} is smooth. }
  \end{rmk}
  \begin{rmk}\emph{
  Still the definition of $n_g^{\beta}$ does not involve 
  virtual classes, so $n_g^{\beta}$ is unlikely to be deformation 
  invariant. On the other hand, our construction of BPS-count
  has a possibility to involve 
  virtual classes using Behrend's constructible functions~\cite{Beh}. 
  As in the proof of Proposition~\ref{dag}, 
  $P(v, \sigma)$ in Definition~\ref{mgd}
  is equal to $\Upsilon_A([(M^v, \rho)])$
  for some projective variety over $A$, $\rho \colon M^v \to A$.
  Then using Behrend's constructible function $\nu\colon 
  M^v \to \mathbb{Z}$, 
  one might try to define $P'(v, \sigma)$ something like 
  $$P'(v, \sigma)=\sum_{n}\pm 
  n\Upsilon_A([(\nu^{-1}(n), \rho|_{\nu^{-1}(n)}]),$$
  and construct $n_g^{\beta '}$ by the same way as in Definition~\ref{mgd}. 
  In this paper we stick to Definition~\ref{mgd} in order to show 
  birational invariance, though 
  it seems interesting to pursue this construction.  }
    \end{rmk}
  
  In the next paragraph, we will show that $n_g ^{\beta}(X)$
  does not depend on a choice of $\omega$. 
    
    \subsection{Local transformation formula of the counting invariants}
    The aim of this paragraph is to give the transformation formula 
    of $\epsilon^{(v, \phi)}(\sigma)$ under 
    small deformations of $\sigma$. 
    Again we assume $X$ is a Calabi-Yau 3-fold. 
    Let us fix the following data, 
    $$v\in N_1(X)\oplus \mathbb{Z}, \quad \phi \in \mathbb{R}, \quad
    \sigma =(Z, \pP) \in \overline{U}_X.$$
    Furthermore we fix an open neighborhood $\sigma \in \uU$
    in $\Stab (X)$ such that $\overline{\uU}$ is compact. 
    We set $\sS \subset \dD _X$ to be the set of objects, 
   $$\sS \cneq \{ E\in \dD _{X} \mid 
   E \mbox{ is semistable in some }
   \sigma '=(Z', \pP')\in \overline{\uU} \mbox{ with }
   \lvert Z'(E) \rvert \le \lvert Z'(v) \rvert \}.$$
   Then $\sS$ has bounded mass, 
   hence there is a wall and chamber structure on 
   $\overline{\uU}$ as in Proposition~\ref{wallcham}. 
     Let $\cC \subset \overline{\uU}$ be a chamber with 
   $\sigma \in \overline{\cC}$. 
   By the definition of the topology on $\Stab (X)$
   (cf.~\cite[Section 6]{Brs1}), 
   we can take $\tau =(W, \qQ) \in U_X \cap \cC$ and $0 <\varepsilon <1/6$ 
   such that 
   $$\pP (\phi) \subset \qQ((\phi -\varepsilon, \phi +\varepsilon)), \quad
   \qQ (\phi) \subset \pP((\phi -\varepsilon, \phi +\varepsilon)),$$
   for any $\phi \in \mathbb{R}$.
   Furthermore we can take $W$ to be defined over $\mathbb{Q}$. 
   Note that for any $v\in C_{\tau}((\phi -\varepsilon, \phi +\varepsilon))$, 
   there are uniquely determined $\phi (v), \phi '(v)
   \in (\phi -\frac{1}{2}, \phi +\frac{1}{2})$
   such that 
   \begin{align}\label{half}W(v) \in \mathbb{R}_{>0}\exp(\pi i \phi(v)),\quad 
   Z(v) \in \mathbb{R}_{>0}\exp(\pi i \phi'(v)),
   \end{align}
   by our choice of $\varepsilon$. Recall the definition 
   of
   $\mathbb{H}\subset
   \mathbb{C}$ in (\ref{up}). 
   \begin{prop}\label{trans}
   Let $\sigma$, $\tau$ and $\varepsilon >0$ be as above. 
   There is a unique sequence of 
   functions $u_n \colon \mathbb{H}^{2n} \to \mathbb{Q}$
   such that we have the following in $\hH(X)$, 
   \begin{align}\label{formula}\epsilon ^{(v, \phi)}(\sigma) &=
   \sum _{v_1 +\cdots +v_n=v}u_n(z_1, \cdots, z_n, w_1, 
   \cdots, w_n)\epsilon^{(v_1, \phi _1)}(\tau)
   \ast \cdots \ast \epsilon ^{(v_n, \phi _n)}(\tau), \\ \label{formula2}
   &=\epsilon ^{(v, \phi)}(\tau)+[\emph{\mbox{ multiple commutators of }}
   \epsilon ^{(v_i, \phi _i)}(\tau)],
    \end{align}
   where $v_i \in C_{\tau}((\phi -\varepsilon, \phi +\varepsilon))$,
   $\phi _i=\phi (v_i)$ and 
    $$z_i=\exp(-\pi i(\phi -1/2))Z(v_i), \quad
    w_i =\exp(-\pi i(\phi -1/2))W(v_i).$$
    Here (\ref{formula}) is a finite sum, and $[\cdots]$ in (\ref{formula2})
    is a finite $\mathbb{Q}$-linear combination of multiple commutators 
    of $\epsilon ^{(v_i, \phi _i)}(\tau)$. 
   \end{prop}
   \begin{proof}
   This is an application of the arguments in~\cite[Theorem 5.2]{Joy4}
   to Bridgeland's stability conditions, and 
   the proof is same as in~\cite[Equation (68)]{Tst3}.
   Note that (\ref{half}) implies $z_i, w_i \in \mathbb{H}$, thus 
   (\ref{formula}) makes sense. 
   By~\cite[Proposition 3.18]{Tst3}, there is $\psi \in \mathbb{R}$ such that
   $$\psi -1< \phi -\varepsilon <\phi +\varepsilon < \psi,$$
   and $\mathfrak{Obj}(\aA _{\psi})\subset \mM$ is open
    for $\aA _{\psi}=\qQ ((\psi -1, \psi])$.
    Since all the terms in (\ref{formula}) are contained in 
    $K_0(\St/\mathfrak{Obj}(\aA_{\psi}))$, it is enough 
    to show (\ref{formula}) in $K_0(\St/\mathfrak{Obj}(\aA_{\psi}))$. 
    Then the straightforward adaptation of the arguments 
    in~\cite[Proposition 5.23]{Tst3}
    which deduces~\cite[Equation (68)]{Tst3} gives the desired 
    equality. (Note that in \textit{loc.cite,} we worked over 
    an algebra $A(\aA _{\psi}, \Lambda, \chi)$, not over 
    $K_0(\St/\mathfrak{Obj}(\aA_{\psi}))$. However readers can find that
    the same argument is applied. Moreover in \textit{loc.cite,}
    the coefficients are given in the form 
    $U(\{v_i\}_{1\le i\le n}, \tau, \sigma)$. These are
    obviously rephrased in terms of the functions 
    $u_n \colon \mathbb{H}^{2n}\to \mathbb{Q}$, 
    as in~\cite[Paragraph 3.1]{Joy}.)   
    Finally the formula (\ref{formula2}) follows from~\cite[Theorem 5.2]{Joy4}.
   \end{proof}
   
   \begin{rmk}\emph{
   Suppose that there is the heart of a t-structure 
   $\aA \subset \dD _X$ with $\mathfrak{Obj}(\aA)\subset \mM$
   open, and all the terms $\epsilon^{(v,\phi)}(\sigma)$,
    $\epsilon^{(v_i,\phi_i)}(\tau)$
   in (\ref{formula}) are contained in $K_0(\St/\mathfrak{Obj}(\aA))$. 
   Then (\ref{formula}), (\ref{formula2}) hold in the algebra
   $K_0(\St/\mathfrak{Obj}(\aA))$. 
   }\end{rmk}
   
   The explicit formula of $u_n$ (cf.~\cite[Definition 4.4]{Joy4})
    is complicated and we do not 
   need this. 
    Now we show the following proposition. 
    \begin{prop}\label{choice}
    For any $v=(\beta,k)\in N_1(X) \oplus \mathbb{Z}$, the
     element $P(v,\sigma)\in \Lambda$ does not depend on 
    a choice of $\sigma \in U_X$. 
    \end{prop}
    \begin{proof}
    For $\sigma \in U_X$, let us take
    $\tau \in U_X$ and $\varepsilon >0$ 
    as in Proposition~\ref{trans}. 
    It is enough to show $P(v,\sigma)=P(v, \tau)$
    in this situation. First assume 
    $v\notin C_{\sigma}(\phi)$ for any $\phi \in \mathbb{R}$, 
    thus $P(v, \sigma)=0$. If $P(v, \tau)\neq 0$, 
    there is some $\tau$-semistable object 
    $E\in \dD _X$ of numerical type $v$. 
    Because $\tau$ is contained in a chamber, 
    $E$ must be also semistable in $\sigma$,
     which is 
    a contradiction.
     (See the comment in~\cite{Brs1} after~\cite[Proposition 8.1]{Brs1}.)
     Hence $P(v, \tau)=0$ follows. 
    
    Next suppose $v\in C_{\sigma}(\phi)$ for some $\phi$. 
    We may assume $0<\phi \le 1$. If $\phi =1$, then $\beta =0$ and 
    $v\in C_{\tau}(1)$. Since we have 
    $$\pP(1)=\qQ(1)=\{\mbox{zero dimensional sheaves}\},$$
    it follows that $\delta ^{(v_i,1)}(\sigma)=\delta ^{(v_i,1)}(\tau)$
    for any $v_i \in C_{\sigma}(1)=C_{\tau}(1)$.  
    Hence $\epsilon ^{(v,1)}(\sigma)=\epsilon ^{(v,1)}(\tau)$
    and $P(v, \sigma)=P(v, \tau)$ follows. 
    
    Finally suppose $0<\phi <1$. We can take $\varepsilon>0$ 
   sufficiently small such that
    $0<\phi -\varepsilon <\phi +\varepsilon <1$. Then all the terms 
    $\epsilon ^{(v_i, \phi _i)}(\tau)$ in (\ref{formula})
    are contained in $K_0(\St/\mathfrak{Coh}(X))$, and $(\ref{formula2})$
    holds in $K_0(\St/\mathfrak{Coh}(X))$. 
    Then applying
    Lemma~\ref{descend} and Remark~\ref{gene}, 
    it is enough to show the following:
    for two varieties $U_1$, $U_2$ with 
   1 morphisms $\rho _i \colon U_i \to 
   \mathfrak{Coh} ^{v_i}(X)$, 
   where 
   $v_i \in C_{\tau}((\phi-\varepsilon, \phi +\varepsilon))$
   with $v_1 +v_2 =v$,
    we have 
   \begin{align}\label{ups}
   \Upsilon _A '(\pi _{\ast}[f_1, f_2])=0, \quad 
   f_i=[(U_i, \rho _i)] \in K_0(\St/\mathfrak{Coh}(X)),\end{align}
   where $A=\Chow _{\beta}(X)$. 
   Note that if there is an exact sequence $0\to E_1 \to E \to E_2 \to 0$
   in $\Coh _{\le 1}(X)$, we have 
   $s(E)=s(E_1 \oplus E_2)$. Using this and 
   Lemma~\ref{gensitu} below, we can conclude (\ref{ups}) holds. 
   \end{proof}
   
   \begin{lem}\label{gensitu}
   Let $\aA \subset \dD _X$ be the heart of a t-structure with 
   $\mathfrak{Obj}(\aA) \subset \mM$ open.
   For $v\in N_1(X)\oplus \mathbb{Z}$, 
    assume that 
   there is a 1-morphism $\pi\colon \mathfrak{Obj}^v(\aA) \to A$, 
   where $A$ is a projective variety, which satisfies,
   \begin{align}\label{dirsum}
   \pi([E])=\pi([E_1 \oplus E_2]), \quad \mbox{for any exact 
   sequence } 0 \to E_1 \to E \to E_2 \to 0 \mbox{ in }\aA,\end{align}
   where $E \in \aA$ is of numerical type $v$. 
   Then 
   we have 
   \begin{align}\label{ups2}\Upsilon _A '(\pi _{\ast}[f_1, f_2])=0, \quad 
   f_i=[(U_i, \rho _i)] \in K_0(\St/\mathfrak{Obj}(\aA)), \end{align}
   where $U_1$, $U_2$ are quasi-projective varieties. 
   \end{lem}
   \begin{proof}
   For $\mathbb{C}$-valued points 
   $p_i \in U_i$, let $E(p_i) \in \aA$ be the 
   objects corresponding to $\rho _i(p_i)$. Let us 
    decompose $U_1 \times U_2$ into finite
   locally closed pieces, 
   $$U_1 \times U_2 =\coprod _k W_k,$$
   such that the dimensions of
   $\Ext ^j(E(p_1), E(p_2))$, 
   $\Ext ^j(E(p_2), E(p_1))$ are 
   constant on each $W_k$
    for $j=0,1$. Furthermore
   we may assume that the bundles
   $$
   \bigcup _{(p_1, p_1)\in W_k}\Ext ^j(E(p_2), E(p_1)) \to W_k,\quad
   \bigcup _{(p_1, p_2)\in W_k}\Ext ^j(E(p_1), E(p_2)) \lr W_k, $$
   are trivial bundles with fibers $V^j _k$, $\bar{V}^{j}_k$ for $j=0,1$
   respectively. 
   Let us consider the diagram (\ref{Hall}) which defines $\ast$-product
   on $K_0(\St/\mathfrak{Obj}(\aA))$. Then 
   the set of $\mathbb{C}$-valued points of the following stack, 
   \begin{align}\label{value}
   W_k \times_{\mathfrak{Obj}(\aA) \times \mathfrak{Obj}(\aA)}
   \mathfrak{Ex}(\aA), \end{align}
   is identified with the $\mathbb{C}$-valued points of 
   $W_k \times V_k^1$. Let 
   $0 \to E(p_1) \to E \to E(p_2) \to 0$
   be an exact sequence in $\aA$ which represents a
    $\mathbb{C}$-valued point of (\ref{value}). 
   Then the stabilizers at this point in (\ref{value}) is 
   identified with  
   the fiber at $(\id, \id)$ of the following morphism, 
   $$\Aut (0 \to E(p_1) \to E \to E(p_2) \to 0) \lr 
   \Aut(E(p_1)) \times \Aut(E(p_2)),$$
   which is isomorphic to $V_k^0=\Hom (E(p_2), E(p_1))$. 
   Hence we have 
   $$(p_1, p_3)^{\ast}\{(f_1, f_2)|_{W_k}\}
   =[W_k \times V^1_k /V^0_k],$$
   where $V^0 _k$ acts on 
   $W_k \times V^1 _k$ trivially. 
   Therefore we can write $f_1 \ast f_2$ in the following form,  
   $$f_1 \ast f_2 =\sum  _k
   [([W_k \times V^1 _k/V^0 _k], \rho ' _k)],$$
   for some 1-morphism $\rho _k' \colon [W_k \times V^1_k/V^0_k] \to 
   \mathfrak{Obj}^v(\aA)$.
    (Also see~\cite[Theorem 5.18]{Joy4}.)
   Let us consider the composition, 
   \begin{align}\label{compo}
   W_k \times V^1 _k \lr [W_k \times V^1 _k/V^0 _k]
   \stackrel{\rho _k '}{\lr}\mathfrak{Obj}^v(\aA) 
   \stackrel{\pi}{\lr} A.\end{align}
   By the assumption (\ref{dirsum}), the above morphism
    is nothing but the following map, 
   $$W_k \times V^1 _k \ni (p_1, p_2, v) \longmapsto \pi([E(p_1) 
   \oplus E(p_2)]).$$
   Hence the morphism (\ref{compo}) descends to the morphism,
   $\rho _k ^{\dag} \colon W_k \to A$, 
   which takes $(p_1, p_2)$ to $\pi([E(p_1) 
   \oplus E(p_2)])$.
   Therefore by Lemma~\ref{fol} and using (\ref{quot}), we have 
   $$\Upsilon _A'(\pi _{\ast}(f_1 \ast f_2))=\sum _k
   \mathbb{L}^{\dim V_k ^1 -\dim V_k ^0}
   \Upsilon _A([(W_k, \rho _k ^{\dag})]).$$
   (Recall that we have defined $\mathbb{L}=
   \Upsilon(\mathbb{A}^1)$ in Definition~\ref{mode}.)
    Arguing as in the same way for $\Upsilon _A'(\pi _{\ast}(f_2 \ast f_1))$
    and taking their difference, we obtain
    $$\Upsilon _A'(\pi _{\ast}[f_1, f_2])=\sum _k
   (\mathbb{L}^{\dim V_k ^1 -\dim V_k ^0}-\mathbb{L}^{\dim \bar{V}_k ^{1} -\dim \bar{V}_k ^{0}})
   \Upsilon _A([(W_k, \rho _k ^{\dag})]).$$
   Then Sublemma~\ref{cy} below shows $\dim V_k ^1 -\dim V_k ^0=
   \dim \bar{V}_k ^{1} -\dim \bar{V}_k ^{0}$, hence 
   (\ref{ups}) follows.  
    \end{proof}
    We have used the following sublemma. 
   \begin{sublem}\label{cy}
   Let $\aA \subset \dD _X$ be the heart of a t-structure and
   take $E, F\in \aA$. Then one has 
   $$\dim \Ext ^1(E,F)-\dim \Hom (E,F)=\dim \Ext ^1(F,E)-\dim \Hom (F,E).$$
   \end{sublem}
   \begin{proof}
   Since $X$ is a Calabi-Yau 3-fold, we have 
   \begin{align*}
   \chi (E,F) &\cneq \sum _{k}(-1)^k \dim \Ext ^k(E,F)\\
   &=-\dim \Ext ^1(E,F)+\dim \Hom (E,F)+\dim \Ext ^1(F,E)-\dim \Hom (F,E),
   \end{align*}
   by Serre duality. On the other hand
   Riemann-Roch implies $\chi (E,F)=0$ because $\ch _i(E)=\ch _i(F)=0$
   for $i=0,1$.  
   \end{proof}
   Combined with Lemma~\ref{dag}, we have the following. 
   \begin{cor}\label{dep}
   For any $\sigma \in U_X$ and $v=(\beta,1)$, 
   we have $P(v,\sigma) =P(v,\sigma _{(0,\omega)})\in \mathbb{Z}[t,s]^{\dag}$, 
   and $n_g^{\beta}(X)$ does not depend on a choice of $\omega$. 
   \end{cor}

   \section{Birational invariance of the counting invariants}\label{Bir}
   Now we state our main theorem. 
   
   \begin{thm}\label{goal}
   Let $\phi \colon W\dashrightarrow X$ be a 
   birational map between smooth projective Calabi-Yau 3-folds. 
   Then for $\beta \in N_1(W)$,
   one has 
   $$n_{g}^{\beta}(W)=n_{g}^{\phi _{\ast}\beta}(X).$$
   \end{thm}
The strategy is as follows. First we enlarge the 
definition of $P(v,\sigma) \in \Lambda$
for some boundary points $\sigma \in \overline{U}_X$, 
 and show $P(v,\sigma)=P(v,\tau)$ for $\tau \in U_X$.  
Next we compare $P(v,\sigma)$ with $P(v',\sigma')$ defined for 
$\sigma ' \in \overline{U}_W$, using the derived equivalence~\cite{Br1}, 
$\Phi \colon \dD_{W}\to \dD_{X}$.

  \subsection{Perverse t-structures on $\dD _X$}
  Before giving the proof of Theorem~\ref{goal}, we investigate 
  some boundary points in $\overline{U}_X$. 
    We assume there is a diagram of birational maps, 
    \begin{align}\label{flop}\xymatrix{
   (C^{\dag}\subset W) \ar[rd]_{g} & & (X \supset C)\ar[dl]^{f} \\
   & (0\in Y), & }\end{align}
   where $C$ and $C^{\dag}$ are tree of rational curves.  
  Furthermore we assume that relative Picard numbers of 
  $f$ and $g$ are one, and $\phi \colon W\dashrightarrow X$
  is not an isomorphism. In this case the diagram (\ref{flop})
  is called a \textit{flop}. (cf.~\cite{Kof}.)
   The main technical tool we use here is the notion of 
   \textit{perverse t-structures} associated to 
   $f\colon X\to Y$. It was introduced by T.~Bridgeland~\cite{Br1}
   to construct the derived equivalence between $W$ and $X$. 
   Below we collect some results we need.
   
  \begin{prop}\label{pt}
  There are hearts of bounded t-structures $\pPPer(\dD_X)\subset 
  \dD_X$ for $p=-1,0$ which satisfy the following. 
  
  (i) For any $E\in \pPPer(\dD_X)$, we have 
  $\dR f_{\ast}E \in \Coh_{\le 1}(Y)$. 
  
   (ii)
  There is an equivalence $\Phi \colon \dD _{W} \to \dD _X$ 
  which restrict to the equivalence, 
  $$\Phi \colon \iPPer (\dD_W) \lr \oPPer (\dD_X).$$
  Furthermore $\Phi$ induce the following commutative diagrams,  
  \begin{align}\label{com3}\xymatrix{
  \dD _W \ar[r]^{\Phi} \ar[d]_{(\ch_2, \ch _3)} & \dD _X 
  \ar[d]^{(\ch_2,\ch_3)}\\
  N_1(W)\oplus \mathbb{Z} \ar[r]^{\phi _\ast} &
   N_1(X)\oplus \mathbb{Z},}
   \quad 
\xymatrix{
\Stab(W) \ar[r]^{\Phi_{\ast}} \ar[d]_{\zZ _W} & \Stab(X) \ar[d]^{\zZ _X} \\
N^1(W)_{\mathbb{C}} \ar[r]^{\phi _{\ast}}& N^1(X)_{\mathbb{C}}.}\end{align} 
  Here $\phi_{\ast}$ in the left diagram takes $(\beta,k)$ to 
  $(\phi_{\ast}\beta,k)$, and $\Phi_{\ast}$ is the natural isomorphism 
  induced by the equivalence $\Phi$.  
  
  (iii) Let $H$ be a relatively ample divisor on $X$ over $Y$.  
  Then for a sufficiently small $0<\delta \ll 1$
  and an ample divisor $\omega'$ on $Y$, the pairs
  \begin{align}\label{pair}\sigma _{(-\delta H, f^{\ast}\omega')}
  =(Z_{(-\delta H, f^{\ast}\omega')}, \oPPer (\dD_X)), 
  \quad \sigma _{(\delta H, f^{\ast}\omega')}
  =(Z_{(\delta H, f^{\ast}\omega')}, \iPPer (\dD_X)),\end{align}
  determine stability conditions contained in $\overline{U}_X$. 
  
  (iv) For $p=-1,0$, the stack of objects $E\in \pPPer(\dD_X)$, 
   $$\fPPer(X) \subset \mM, $$
 is an open substack of $\mM$.

    \end{prop}
    \begin{proof}
    In~\cite[Section 3]{Br1}, Bridgeland constructed the hearts of 
 some bounded t-structures
 $\pPPer(X/Y)$ on $D(X)$ for $p=-1,0$. 
 For simplicity we discuss the case of $p=0$. 
According to~\cite[Lemma 3.1]{MVB}, the
abelian category $\oPPer(X/Y)$ is obtained from $\Coh(X)$ as 
a tilting of the torsion pair, 
\begin{align*}
\tT _0 &=\{T\in \Coh(X) \mid R^1f_{\ast}T=0\}, \\
\fF_0 &=\{ F\in \Coh(X) \mid f_{\ast}F=0, \Hom(\mathfrak{C},F)=0\},
\end{align*}
where $\mathfrak{C}\cneq \{E\in\Coh(X) \mid \dR f_{\ast}E=0\}$,
i.e. $\oPPer(X/Y)$ is generated by $\fF _0[1]$ and $\tT _0$. 
 Let us define $\pPPer(\dD_X)$ to be
 $$\pPPer(\dD_X)\cneq \dD_X \cap \pPPer(X/Y).$$
 We have to check that $\pPPer(\dD_X)$ is the heart of a bounded t-structure 
 on $\dD_{X}$.
Since $\fF _0 \subset \Coh_{\le 1}(X)$, 
the pair $(\fF_0, \tT_0 \cap \Coh_{\le 1}(X))$ also
determines a torsion pair on $\Coh_{\le 1}(X)$,
and the corresponding tilting is $\oPPer(\dD_X)$. 
Thus $\oPPer(\dD_X)$ is the heart of a bounded t-structure on $\dD _X$. 
(cf.~\cite{HRS}.)

\hspace{3mm}
    
  (i) For $E\in \pPPer(\dD_X)$, the object 
$\dR f_{\ast}E$ must be a sheaf by the definition of 
$\pPPer(X/Y)$ in~\cite[Section 3]{Br1}. 

\hspace{3mm}

(ii) In~\cite{Br1}, Bridgeland constructed the equivalence, 
$$\Phi \colon D(W) \lr D(X),$$
which restricts to an equivalence between $\iPPer(W/Y)$ and 
$\oPPer(X/Y)$. Furthermore Chen~\cite{Ch}
showed that $\Phi$ is given by a Fourier-Mukai functor with kernel 
$\oO_{W\times _Y X}$. Because $\phi \colon W\dashrightarrow X$
is an isomorphism in codimension one, the equivalence 
$\Phi$ takes $\dD _{W}$ to $\dD_{X}$.

 For the left diagram of (\ref{com3}), 
let us take a divisor $D$ on $X$ and 
$E\in \dD_{W}$. By Riemann-Roch theorem, we have 
\begin{align}\label{ri1}
\chi (\oO_X(D), \Phi(E)) &= -D\cdot \ch _2\Phi(E) +\ch _3 \Phi(E), \\
\label{ri2}
\chi (\Phi^{-1}\oO_X(D), E) &=-\phi_{\ast}^{-1}D\cdot \ch _2(E) +\ch _3(E).
\end{align}
Here we have used the fact that $\ch _1\Phi^{-1}\oO_X(D)=\phi_{\ast}^{-1}D$.
This follows because $\Phi^{-1}(\oO_X(D))$ and $\oO_X(\phi_{\ast}^{-1}D)$
are isomorphic over $W\setminus C^{\dag}$, and $C^{\dag}$ has 
codimension two in $W$. (cf.~\cite[Lemma 3.15]{Tst}.)
By adjunction we must have (\ref{ri1})$=$(\ref{ri2}), and 
this holds for any divisor $D$. Thus we have 
$$(\ch _2 \Phi(E), \ch _3 \Phi(E))=(\phi_{\ast}\ch_2 (E), \ch_3 (E)),$$
by the definition of $\phi_\ast \colon N_1(W) \to N_1(X)$. 
  
   For the commutativity of the right diagram of (\ref{com3}), 
  the same proof of~\cite[Lemma 4.8]{Tst} is applied, and 
  we leave the readers to check the detail. 
   
   \hspace{3mm}

   (iii) The same proof of~\cite[Lemma 4.3]{Tst}
shows that the pairs (\ref{pair}) give stability conditions. 
In fact arguing as in~\cite[Lemma 3.8 (iii)]{Tst}, any object 
in $\oPPer(\dD_X)$ is given by a successive extension of the 
following objects, 
\begin{align}\label{gen1}
& S_0=\omega_{f^{-1}(0)}[1], \quad S_i=\oO_{C_i}(-1) \ (1\le i\le m), \\
\label{gen2}
& \widetilde{\Coh}_{\le 1}(X)\cneq  \{F\in \Coh_{\le 1}(X) \mid C_i \nsubseteq \Supp(F)
\mbox{ for all }i
\}.
\end{align}
Here $C_i$ for $1\le i\le m$ are the irreducible components of $C$
and $f^{-1}(0)$ is the scheme theoretic fiber of $f$ at $0\in Y$. 
In order to show (\ref{up}) in Proposition~\ref{equiv},
it is enough to check this for 
the generators (\ref{gen1}), (\ref{gen2}). 
For $Z=Z_{(-\delta H, f^{\ast}\omega ')}$, we have  
\begin{align*}
&Z(S_0)= -1 +\delta H \cdot f^{-1}(0)<0, \quad 
Z(S_i)=-\delta H\cdot C_i <0 \ (1\le i\le m), \\
& \Imm Z(F)>0 \quad \mbox{ for } \ F\in \widetilde{\Coh}_{\le 1}(X)\setminus \{0\},
\end{align*}
thus (\ref{up}) holds. 
We leave the readers to check the
Harder-Narasimhan property, applying 
the proof of~\cite[Lemma 4.3]{Tst}. Also the case of $p=-1$ is similarly 
proved. Finally we have to check 
that the stability conditions determined by (\ref{pair}) are 
contained in $\overline{U}_X$. Since it requires some 
more technical arguments, we postpone it until Section~\ref{tech}.

 \hspace{3mm}
 
  (iv) According to~\cite{MVB}, 
 there are vector bundles ${^{p}\eE}$ on $X$ for $p=-1,0$
 such that
 an object 
 $E\in \dD_X$ is contained in $\pPPer(\dD_X)$
 if and only if 
 \begin{align}\label{condi}
 \dR f_{\ast}\dR \hH om({^{p}\eE}, E) \in \Coh(Y).
 \end{align}
 Since (\ref{condi}) is an open condition, the stack 
 $\fPPer(X)$ is an open substack of $\mM$.

    \end{proof}

  For $v=(\beta, k)\in N_1(X) \oplus \mathbb{Z}$, let 
   $$\fPPer^{v}(X)\subset \fPPer(X),$$
    the substack
   of objects $E\in \pPPer(\dD_X)$
   of numerical type $v$. By Proposition~\ref{pt} (i),
   we have the 1-morphism, 
   $$\bar{\pi}\colon \fPPer^{v}(X) \ni E 
   \longmapsto s(\dR f_{\ast}E) \in \Chow _{f_{\ast}\beta}(Y).$$
  Let $\sigma \in \overline{U}_X$ be one of (\ref{pair}), 
  corresponding to $\pPPer(\dD_X)$ for $p=-1$ or $0$. 
  As in (\ref{sf}), for $0<\phi \le 1$ we have 
  \begin{align}\label{again}
  \epsilon ^{(v, \phi)}(\sigma) \in K_0(\St/\fPPer(X)).
  \end{align}
   \begin{defi}\label{abo}\emph{
   Let $\sigma \in \overline{U}_X$ be one of (\ref{pair}). We
   define $P(v, \sigma)\in \Lambda$ as follows. }
   \begin{itemize}
   \item \emph{If $v\in C_{\sigma}(\phi)$ with 
   $0<\phi \le 1$,
   we define 
   $$P(v, \sigma)\cneq (\mathbb{L}-1)\Upsilon _{\bar{A}} '
   (\bar{\pi}_{\ast}\epsilon ^{(v,\phi)}(\sigma)),$$
   where 
   $\bar{A}=\Chow _{f_{\ast}\beta}(Y)$.
   By (\ref{again}), $\bar{\pi}_{\ast}
   \epsilon ^{(v,\phi)}(\sigma) \in K_0(\St/\bar{A})$ makes sense.}
   \item \emph{If $v\in C_{\sigma}(\phi)$ with 
   $1<\phi \le 2$,
   we define  
   $$P(v, \sigma)\cneq P(-v, \sigma).$$}
   \item \emph{Otherwise we 
   define $P(v, \sigma)=0$. }
   \end{itemize} 
   \end{defi}
   
   \begin{rmk}\emph{
   In~\cite{MVB}, it is shown that
   there are vector bundles ${^{p}\eE}$ on $X$ for $p=-1,0$
   such that there are equivalences,
   $$ \pPPer(X/Y) \stackrel{\sim}{\lr}
   \Coh(f_{\ast}\hH om({^{p}\eE}, {^{p}\eE})). $$
   Since the RHS is the module category over a non-commutative sheaf
   of algebras on $Y$, BPS counting
   constructed from Definition~\ref{abo} and the formula (\ref{gvde}) 
   is interpreted as (approximation of) non-commutative 
   Gopakumar-Vafa invariant. 
   It seems interesting to pursue its relationship to 
   non-commutative Donaldson-Thomas theory on conifold 
   studied by B.~Szendr{\H o}i~\cite{Sz}.}   
   \end{rmk}
   For $\sigma \in \overline{U}_X$ as in Proposition~\ref{pt} (iii),
    let us
   write it $\sigma =(Z, \pP)$ as in Definition~\ref{stde}.
   We also use the abelian category
   $$\aA_{\sigma} =\pP(( -1/2, 1/2]),$$
   and the stack of objects in $\aA_{\sigma}$, $\mathfrak{Obj}(\aA_{\sigma})\subset \mM$. 
   By Proposition~\ref{pt} (iv) and~\cite[Proposition 3.18]{Tst3}, 
   $\mathfrak{Obj}(\aA_{\sigma})$ is an open substack of $\mM$.  
   We have the following lemma. 
   \begin{lem}\label{comm}
   Under the above situation, the following diagram is commutative,  
   \begin{align}\label{com1}\xymatrix{
  K_0(\St/\mathfrak{Coh}^v(X)) \cap 
  K_0(\St/\fPPer^v(X)) \ar[r] \ar[dr]
  & K_0(\St/\mathfrak{Coh}^v(X)) \ar[r]^{\pi_{\ast}} & 
   K_0(\St/A) \ar[rd]^{\Upsilon _A'} \ar[d]_{f_{\ast}} & \\
   & K_0(\St/\fPPer^v(X)) \ar[r]^{\bar{\pi}_{\ast}}
    & K_0(\St/\bar{A})
   \ar[r]^{\Upsilon _{\bar{A}}'} & \Lambda.}\end{align}
   Here $A=\Chow_{\beta}(X)$ and $\bar{A}=\Chow_{f_{\ast}\beta}(Y)$. 
   Furthermore if $f^{\ast}\omega' \cdot \beta =0$, the 
   following diagram commutes,
   \begin{align}\label{com2}
   \xymatrix{
   K_0(\St/\mathfrak{Coh}^v(X)) \cap 
  K_0(\St/\mathfrak{Obj}^v(\aA_{\sigma})) \ar[r] \ar[dr]
  & K_0(\St/\mathfrak{Coh}^v(X)) \ar[r]^{\pi_{\ast}} & 
   K_0(\St/A) \ar[rd]^{\Upsilon _A'} \ar[d]_{\pi_{1\ast}} & \\
   & K_0(\St/\mathfrak{Obj}^v(\aA_{\sigma})) \ar[r]^{\pi_{0\ast}}
    & K_0(\St/\mathbb{C})
   \ar[r]^{\Upsilon'} & \Lambda.}\end{align}
   Here $\pi _0 \colon \mathfrak{Obj}^v(\aA_{\sigma}) \to \Spec \mathbb{C}$, 
   $\pi _1\colon A\to \Spec \mathbb{C}$ are the
   structure morphisms. The diagram also commutes after 
   the following replacements. 
   \begin{align*}
   \begin{array}{lll}
  \mathfrak{Coh}^v(X) \mapsto \mathfrak{Coh}^v(X)[1], &
 \pi \mapsto \pi', & A \mapsto A', \\
\mathfrak{Coh}^v(X) \mapsto \mathfrak{Per}^v(X), &
   \pi \mapsto \bar{\pi}, & A \mapsto \bar{A}.
   \end{array}
   \end{align*}
   Here we have used the notation of Remark~\ref{observe}. 
   \end{lem}
   \begin{proof}
   In both diagrams, the commutativity of the 
   LHS follows from the functorial property of the push-forwards. 
   Let us show the commutativity of the RHS. 
   Since $f\colon X\to Y$ contracts only finite number of 
   rational curves, the map 
   $$f_{\ast}\colon A \ni Z \longmapsto f_{\ast}Z \in \bar{A},$$
   which sends an algebraic cycle $Z$ on $X$ to the cycle 
   $f_{\ast}Z$ on $Y$, is a finite morphism. 
   (See~\cite[Theorem 6.8]{Ko} for the existence of 
   the above morphism.)
   Hence 
   we can apply Lemma~\ref{finite}, which shows the commutativity 
   of the RHS of (\ref{com1}).
    Finally
   if $f^{\ast}\omega'\cdot \beta=0$, any effective one cycle of homology 
   class $\beta$ is contracted by $f$. 
   Thus $\bar{A}=\Spec \mathbb{C}$, and $f_{\ast}\colon A\to \bar{A}$
    is identified with 
   $\pi_1$.  Therefore the commutativity of 
   the RHS of (\ref{com2}) also follows from 
   Lemma~\ref{finite}.  
   \end{proof}
   Now we show the following proposition. 
   \begin{prop}\label{bound}
   Let $\sigma \in \overline{U}_X$ be as above  
   and $\tau =(W, \qQ)\in U_X$. Then for any $v\in N_1(X)\oplus \mathbb{Z}$, 
   one has $P(v,\sigma)=P(v,\tau)$. 
   \end{prop}
   \begin{proof}
   It is enough to show $P(v, \sigma)=P(v,\tau)$
   under the situation of Proposition~\ref{trans}. 
   Furthermore the same proof of Proposition~\ref{choice}
   shows $P(v, \sigma)=P(v,\tau)=0$ if 
   $v\notin C_{\sigma}(\phi)$ for any $\phi \in \mathbb{R}$. 
   Thus we may assume $v\in C_{\sigma}(\phi)$ 
   for some $0<\phi \le 1$. 
   First we assume $0<\phi <1$. Let us take
    $\varepsilon >0$ as in Proposition~\ref{trans}. 
  We can choose $\varepsilon>0$ sufficiently small so that 
   $0<\phi -2\varepsilon <\phi +2\varepsilon <1$. Then 
   we have 
   $$\qQ((\phi -\varepsilon, \phi +\varepsilon))
   \subset \pPPer(\dD_X)\cap \Coh _{\le 1}(X).$$
   Thus all the terms in (\ref{formula}) are contained in
   both $K_0(\St/\fPPer(X))$ and $K_0(\St/\mathfrak{Coh}(X))$, 
   and (\ref{formula2}) holds in both algebras. 
   Applying Lemma~\ref{gensitu} for $\aA =\Coh _{\le 1}(X)$, we have 
   \begin{align}\label{eq1}
   (\mathbb{L}-1)\Upsilon _{A}'(\pi_{\ast}\epsilon^{(v, \phi)}(\sigma))
   =(\mathbb{L}-1)\Upsilon _{A}'(\pi_{\ast}\epsilon^{(v, \psi)}(\tau)),
   \end{align}
   for some $\psi \in (\phi -\varepsilon, \phi +\varepsilon)$. 
   (Here we have used the same notation in Lemma~\ref{comm}.)
   By Definition~\ref{mode}, the RHS of (\ref{eq1})
   is $P(v,\tau)$. On the other hand
   we have 
   \begin{align}\label{eq2}(\mathbb{L}-1)\Upsilon _{A}'(\pi_{\ast}\epsilon^{(v, \phi)}(\sigma))
   =(\mathbb{L}-1)
   \Upsilon _{\bar{A}}'(\bar{\pi}_{\ast}\epsilon^{(v, \phi)}(\sigma)),
   \end{align}
   by the diagram (\ref{com1}).
   Then the RHS of (\ref{eq2})
   is $P(v,\sigma)$ by Definition~\ref{abo}. Hence (\ref{eq1}) and (\ref{eq2}) show
   $P(v, \sigma)=P(v, \tau)$. 
   
   Next suppose $\phi =1$. 
   Note that in this case $f^{\ast}\omega' \cdot \beta=0$
   for $v=(\beta, k)$. 
   For a sufficiently small $\varepsilon >0$,
   one has 
   $$\qQ((1 -\varepsilon, 1+\varepsilon))\subset \aA_{\sigma}.$$
   Thus the formulas
   (\ref{formula}), (\ref{formula2}) hold 
   in $K_0(\St/\mathfrak{Obj}(\aA_{\sigma}))$. Applying Lemma~\ref{gensitu} for $\aA =\aA _{\sigma}$, 
    we have 
   \begin{align}\label{Note}
   (\mathbb{L}-1)\Upsilon'(\pi_{0\ast}\epsilon^{(v,1)}(\sigma))
   =(\mathbb{L}-1)
   \Upsilon'(\pi_{0\ast}\epsilon^{(v,\psi)}(\tau)),\end{align}
   for some $\psi \in (1-\varepsilon, 1+\varepsilon)$. 
   If $\psi \le 1$, 
   then $\epsilon^{(v,\psi)}(\tau)$ is contained in 
   both $K_0(\St/\mathfrak{Coh}(X))$ and 
   $K_0(\St/\mathfrak{Obj}(\aA_{\sigma}))$, 
   hence the RHS of (\ref{Note}) is equal to $P(v, \tau)$ by 
   the diagram (\ref{com2}).
   Also the
   diagram (\ref{com2}) after the following replacement, 
    $$\mathfrak{Coh}^v(X) \mapsto \fPPer^{v}(X), 
    \quad \pi \mapsto \bar{\pi} \quad 
    A \mapsto \bar{A}, $$
  shows that 
   the LHS of (\ref{Note}) is equal to $P(v, \sigma)$. 
   Hence in this case we obtain $P(v, \sigma)=P(v, \tau)$. 
   When $\psi >1$, we can use the diagram (\ref{Note})
  after the replacement, 
  $$\mathfrak{Coh}^v(X) \mapsto \mathfrak{Coh}^v(X)[1], 
    \quad \pi \mapsto \pi' \quad 
    A \mapsto A', $$
    and conclude that $P(v, \sigma)=P(v, \tau)$. 
  In fact by Remark~\ref{observe}, we see that the RHS of (\ref{Note}) 
  is equal to $P(v, \tau)$ also in this case.
    \end{proof}

  \subsection{Proof of Theorem~\ref{goal}}
   \begin{proof}
   We may assume that there is a diagram (\ref{flop}), 
   since any birational map $\phi \colon W\dashrightarrow X$
   is connected by a sequence of flops~\cite{Kawaflo}. 
   Let $H$ be a relatively ample divisor on $W$ over $Y$ and 
   $$\sigma=(Z_{(-\delta H, g^{\ast}\omega')}, \iPPer(\dD_W))
   \in \overline{U}_W,$$
   the stability condition in (\ref{pair}), 
   applied for $g\colon W\to Y$. 
   By Proposition~\ref{pt} (ii),  we have 
   $$\Phi_{\ast}\sigma=( 
    Z_{(-\delta \phi_{\ast}H, f^{\ast}\omega')}, \oPPer(\dD_X)).$$
    (Here we have used the right diagram of (\ref{com3}).)
    Since $-\phi_{\ast}H$ is relatively ample over $Y$, 
    $\Phi_{\ast}\sigma$ is one of the stability 
    conditions constructed in (\ref{pair}). Hence by 
    Proposition~\ref{bound}, it is enough to show 
    \begin{align}\label{wall}
    P(W, v, \sigma)=P(X, \phi_{\ast}v, \Phi_{\ast}\sigma).
    \end{align}
   We may assume $v\in C_{\sigma}(\phi)$ for $0<\phi \le 1$. 
   By Proposition~\ref{pt} (ii), 
   the equivalence $\Phi$ induces the isomorphism, 
   \begin{align}\label{In}\Phi_{\St}\colon  K_0(\St/\ifPPer(W)) 
   \lr K_0(\St/\ofPPer(X)). \end{align}
   It is easy to see that (\ref{In}) preserves $\ast$-product.
   Furthermore by the left diagram of (\ref{com3}), the isomorphism 
   (\ref{In}) takes $\delta^{(v, \phi)}(\sigma)$ to
    $\delta^{(\phi_{\ast}v, \phi)}(\Phi_{\ast}
   \sigma)$.
    Thus we have
   \begin{align}\label{indu}
   \Phi_{\St}\epsilon^{(v,\phi)}(\sigma)=\epsilon^{(\phi_{\ast}v,\phi)}(
   \Phi_{\ast}\sigma).\end{align}
   On the other hand, we have the 
   commutative diagram, 
   \begin{align}\label{final}\xymatrix{
   K_0(\St/\ifPPer^v(W)) \ar[r]^{\Phi_{\St}}\ar[rd]_{\bar{\pi}_{W\ast}} & 
   K_0(\St/\ofPPer^{\phi_{\ast}v}(X))\ar[d]_{\bar{\pi}_{X\ast}} & \\
   & K_0(\St/\bar{A}) \ar[r]^{\Upsilon_{\bar{A}}} & \Lambda.}\end{align}
   Hence the diagram (\ref{final}) together with (\ref{indu}) 
   imply (\ref{wall}). 
   \end{proof}
   
    \begin{rmk}\label{betti}\emph{
    If $v=(0,1)$ and $\sigma \in U_X$, 
     then $P(v,\sigma)$ is equal to $\sum _{i}b_i(X) t^i$. 
     Hence Proposition~\ref{bound} and (\ref{wall}) imply 
    $b_i(W)=b_i(X)$ for all $i\in \mathbb{Z}$. }
     \end{rmk}
  
  \begin{exam}\emph{
  Let $\mathbb{P}^1 \cong C\subset X$ be a 
  $(-1,-1)$ curve, i.e. the normal bundle $N_{C/X}$ 
  is isomorphic to $\oO_{C}(-1)\oplus \oO_{C}(-1)$. 
  For $m\neq 0$, the same computation in~\cite[Proposition 4.5]{HST} shows, 
  $$n_g^{m[C]}(X)= \left\{
  \begin{array}{ll} 1 & \mbox{if }g=0, m=\pm 1, \\
  0 & \mbox{otherwise.} \end{array} \right. $$
  Let $\phi \colon W\dashrightarrow X$ be a flop at $C$, 
  and $C^{\dag}\subset W$ the flopped curve. 
  Since $\phi_{\ast}[C^{\dag}]=-[C]$, one can see 
   $$n_g^{[C^{\dag}]}(W)=n_g ^{-[C]}(X)=1,$$
  by the above computation. Note that since $-[C]$ is not effective,
  the invariant $n_g^{\beta}(X)$ should also be defined for non-effective 
  one cycle classes $\beta$. }
  \end{exam}

  \section{Some technical lemmas}\label{tech}
  In this section, we prove some postponed technical lemmas. 
  \subsection{Proof of Lemma~\ref{op}}
  \begin{proof}
  Let us take
  $\sigma=\sigma_{(B,\omega)}\in U_X$ and $B'+i\omega '
  \in N^1(X)_{\mathbb{C}}$. 
  First we show 
  $\lVert Z_{(B, \omega)} -Z_{(B', \omega')}\rVert _{\sigma}<\infty$. 
  (See~\cite[Section 6]{Brs1} for $\lVert \ast \rVert _{\sigma}$.)
  By the definition of $\lVert \ast \rVert _{\sigma}$, it is equal to
  \begin{align}\label{sup}
  \sup \left\{ 
  \frac{\lvert\{(B-B')+i(\omega-\omega')\}\ch_2(E)\rvert}{\lvert Z_{(B,\omega)}(E)\rvert} : E \mbox{ is semistable in } 
  \sigma \right\}.
  \end{align}
  Let us put $m=\lvert Z_{(B, \omega)}(E) \rvert$
  for a $\sigma$-semistable object $E$. Then we have 
  $\lvert\omega \cdot \ch_2(E)/m \rvert \le 1$. 
 We set $K \subset N_1(X)$ as
   \begin{align}\label{comp}
  K\cneq \{ c\in \overline{\NE}(X) \mid \omega \cdot c \le 1\}
  \subset N_1(X).\end{align}
  As in the proof of Proposition~\ref{wallcham}, the space
  $K$ is compact. 
  Therefore the function 
  $$K \ni c \longmapsto \lvert \{(B-B')+
  i(\omega-\omega')\}c\rvert \in \mathbb{R}$$
  has a maximum value, say $M$. Since $\ch _2(E)/m \in K$
  or $-\ch_2(E)/m \in K$, 
  we have $(\ref{sup})\le M <\infty$. 
  
  Then by~\cite[Proposition 6.3]{Brs1}, the 
  map $\Stab(X) \to N^1(X)_{\mathbb{C}}$ is a local
  homeomorphism. 
 Suppose that
   $B'+i\omega'$ satisfies
  $$\lVert Z_{(B, \omega)}- Z_{(B', \omega')} \rVert _{\sigma} <
  \sin \pi\varepsilon,$$
  for a sufficiently small $\varepsilon$. 
  Then~\cite[Theorem 7.1]{Brs1} guarantees the 
  existence of a stability condition $\tau=(Z_{(B',\omega')}, \qQ)$
  which satisfies $d(\pP, \qQ)<\varepsilon$. (See~\cite[Section 6]{Brs1} for 
  $d(\ast,\ast)$.) If we know $\tau \in U_X$, we can conclude 
  $U_X$ is open. 
  
  To conclude $\tau \in U_X$, it is enough to check 
  $\qQ((0,1])\subset \Coh_{\le 1}(X)$. 
  According to the proof of~\cite[Theorem 7.1]{Brs1}, 
  the set of objects $\qQ(\phi)$ for $0<\phi \le 1$
  is obtained as follows:
  an object $E\in \dD _X$ is contained in $\qQ(\phi)$ if and only if 
  there is a thin and enveloping subcategory $E\in \pP((a,b))$
 such that $E$ is $Z_{(B',\omega')}$-semistable with phase $\phi$. 
  (See~\cite[Definition 7.2, Definition 7.4]{Brs1} 
  for the notion of thin enveloping
  subcategory.) Take $E\in \qQ(\phi)$ with $E\in \pP((a,b))$
  as above. If $0<a<b\le 1$, one has $E\in \Coh_{\le 1}(X)$. 
  Suppose $b>1$. Then there is a distinguished 
  triangle 
 \begin{align}\label{tr}
 H^{-1}(E)[1] \lr E \lr H^0(E),\end{align}
  with $H^{-1}(E)[1] \in \pP((1,b))$ and $H^0(E)\in \pP((a,1])$. 
  The semistability of $E$ in $Z_{(B',\omega')}$ implies 
  \begin{align}\label{arg}
  \arg Z_{(B',\omega')}(H^{-1}(E)[1]) \le \arg Z_{(B',\omega')}
  (H^0(E)).\end{align}
  Here $\arg$ is taken in the interval $(\pi i(a-\varepsilon), 
  \pi i(b+\varepsilon))$. However since $H^i(E) \in \Coh_{\le1}(X)$, 
  (\ref{arg}) implies $E\cong H^{0}(E)$ or $E\cong H^{-1}(E)[1]$. 
  Since $E$ has phase $0<\phi \le 1$ with respect to 
  $Z_{(B',\omega')}$, one must have $E\cong H^{0}(E)\in \Coh_{\le 1}(X)$. 
  The similar argument shows $E\in \Coh_{\le 1}(X)$ when 
  $a\le 0$.  
  \end{proof}
  
  \subsection{Proof of Lemma~\ref{bounded}}
  \begin{proof}
  In fact we show the following stronger claim. 
  Let $\Sigma \subset \Stab(\dD_{X})$ be the 
  connected component which contains $U_X$, and 
  $\Sigma'\subset \Sigma$ the subset of $\sigma=(Z, \pP)$
  such that $\Imm Z\subset \mathbb{C}$ is discrete. 
  We show that for any $\sigma \in \Sigma'$, $\phi \in \mathbb{R}$
  and $z\in \mathbb{C}$, the following set of objects, 
  $$M^{(z, \phi)}(\sigma)\cneq \{ E\in \pP(\phi)\mid Z(E)=z\},$$
  is bounded. 
   For an ample divisor $\omega$ on $X$, let us consider the point 
   $\sigma_{(0, \omega)} \in U_X$. By Remark~\ref{Gs}, 
   any $\sigma_{(0, \omega)}$-semistable object is
   nothing but $\omega$-Gieseker semistable sheaf up to shift.
   As is well-known, the set of $\omega$-Gieseker semistable 
   sheaves with a fixed Hilbert polynomial 
   forms a bounded family, hence 
   the claim is true for $\sigma=\sigma_{(0, \omega)}$. 
   
    Next suppose that the above claim is true for some 
   $\sigma =(Z, \pP) \in \Sigma'$. 
    We show that if $\tau =(W, \qQ)
   \in \Sigma'$ is sufficiently close to $\sigma$, 
   then the claim also holds for $\tau$. 
   Obviously once we show this, then the claim is true 
   for any $\sigma \in \Sigma'$. 
   In order to show the boundedness of $M^{(z, \phi)}(\tau)$, 
   we may assume $\phi =1/2$ by applying some element 
   $g\in \widetilde{\GL}^{+}(2, \mathbb{R})$
   (cf.~\cite[Lemma 8.2]{Brs1}) to $\sigma$, $\tau$. 
   If $\tau$ is sufficiently close to $\sigma$, we have 
   $$\qQ(1/2) \subset \pP((1/4, 3/4)) \subset  \qQ((0,1)).$$
    For $E\in M^{(z, 1/2)}(\tau)$, 
   let $F_i \in \pP(\phi_i)$ for $\phi _i \in (1/4, 3/4)$, $1\le i\le n$
   be the $\sigma$-semistable factors of $E$. 
   Since $\Imm W(F_i) \le \Imm W(E)$, there is $\delta >0$
   which does not depend on $E$ such that $\Imm Z(F_i) \le \Imm z +\delta$,
   if $\tau$ is enough close to $\sigma$. 
   Because $\Imm Z \subset \mathbb{C}$ is discrete, 
    we see that numbers of semistable factors $n$, and
   the values $z_i=Z(F_i) \in \mathbb{C}$ have finite number of 
   possibilities. Since $F_i \in M^{(z_i, \phi_i)}(\sigma)$, 
   the boundedness of $M^{(z_i, \phi_i)}(\sigma)$ for each $i$ 
   implies the boundedness of $M^{(z, \phi)}(\tau)$. 
   \end{proof}

  \subsection{Proof of Lemma~\ref{fin}}
  \begin{proof}
  We may assume $0<\phi \le 1$ and
  let us take $0<\varepsilon <1/6$. 
  Since $\sigma \in \overline{U}_X$, there is 
  $\tau =(Z_{(B,\omega)},\Coh_{\le 1}(X)) \in U_X$
  with $B$, $\omega$ rational  
  such that $C_{\sigma}(\phi) \subset C_{\tau}((\phi -\varepsilon, \phi+\varepsilon))$. From this it is clear that there is a finite number of
  possibilities for $n$ in (\ref{eps}). Hence it is enough to 
  check the finiteness of the set,
  $$\{(v_1, v_2) \mid v_1+v_2=v, v_i \in C_{\sigma}(\phi)\}.$$
  We write $v_i=(\beta_i,k_i) \in N_1(X) \oplus \mathbb{Z}$. 
  It is enough to check that the possible pairs
   $(\beta_1,\beta_2)$ are finite.  
  First we assume $0<\phi <1$. We may assume
  that $0<\phi-\varepsilon <\phi +\varepsilon <1$. 
  Then $\beta _i \in \overline{NE}(X)$ and we have 
  \begin{align}\label{le}
  \Imm Z_{(B,\omega)}(v_i) \le \Imm Z_{(B,\omega)}(v).\end{align}
  Since (\ref{le}) implies $\beta _i \cdot \omega \le \beta \cdot \omega$
  and (\ref{comp}) is compact, the possible pairs 
  $(\beta _1, \beta _2)$ must be finite. 
  
  Next we treat the case of $\phi =1$. Then $v_i$ is decomposed
  as follows, 
  $$v_i =\sum _{j}v_{ij}, \quad v_{ij} \in C_{\tau}(\phi_{ij}) \mbox{ with }
  \phi _{ij}\in (1-\varepsilon, 1+\varepsilon).$$
  If we write $v_{ij}=(\beta_{ij},k_{ij})$, 
  then $\beta_{ij} \in \overline{NE}(X)$ or 
  $-\beta_{ij}\in \overline{NE}(X)$. 
   We can easily see, 
  $$\lvert \beta _{ij}\cdot \omega \rvert =\lvert \Imm Z_{(B,\omega)}(v_{ij}) \rvert \le \lvert 
  \Ree Z_{(B,\omega)}(v) \rvert \cdot \tan \pi \varepsilon.$$
  Again since (\ref{comp}) is compact, the possible 
  $\{\beta _{ij}\}_{i,j}$ are finite. Thus the pair
  $(\beta _1, \beta _2)$ 
  also has a finite number of possibilities.

  \end{proof}
  
 \subsection{Proof of Proposition~\ref{pt} (iii)}
 \begin{proof}
 We have to show stability conditions in (\ref{pair}) 
are contained in $\overline{U}_X$. In~\cite[Proposition 4.4]{Tst}, 
the author put the assumption that there exists a hyperplane 
$Y_0\subset Y$ such that $f^{-1}(Y_0)$ is smooth. 
In our purpose, we have to improve the proof and show that
 actually stability conditions in (\ref{pair})
are contained in $\overline{U}_X$ without such assumption. 

The proof goes on as in Lemma~\ref{op}, and we show
 the case of $p=0$ for simplicity. 
Let $\sigma =\sigma_{(-\delta H, f^{\ast}\omega')}$ be 
as in (\ref{pair}), and take $B+i\omega \in N^1(X)_{\mathbb{C}}$. 
We also set $Z=Z_{(-\delta H, f^{\ast}\omega')}$. 
The value $\lVert Z-Z_{(B,\omega)}\rVert _{\sigma}$ is given by 
\begin{align}\label{go}
\sup \left\{ 
  \frac{\lvert\{(-\delta H-B)+i(f^{\ast}\omega'-\omega)\}\ch_2(E)\rvert}{\lvert Z(E)\rvert} : E \mbox{ is semistable in } 
  \sigma \right\}.\end{align}
In order to show (\ref{go}) is finite, it is enough 
to give the upper bound of (\ref{go}) for 
$E\in \oPPer(\dD_X)$.
Let us take $F\in \widetilde{\Coh}_{\le 1}(X)$ where $\widetilde{\Coh}_{\le 1}(X)$ is given by (\ref{gen2}), 
 and 
 put $m=\lvert Z(F)\rvert$. 
 Then $f^{\ast}\omega '\cdot \ch _2(F)/m \le 1$. 
 By the openness of the big cone, there is a sufficiently small 
 rational polyhedral cone
 $f^{\ast}\omega ' \in \Delta \subset \phi_{\ast}\overline{A}(W)\cup 
 \overline{A}(X)$. Let $K'$ be 
 $$K' \cneq \{ c\in \check{\Delta} \mid f^{\ast}\omega '\cdot c\le 1\}\subset
 N_1(X),$$
 where $\check{\Delta}$ is the dual cone. 
 Then $K'$ is compact, hence the function 
 $$K' \ni c \longmapsto  \lvert \{(-\delta H-B)+
 i(f^{\ast}\omega '-\omega)\} \cdot c \rvert \in \mathbb{R}$$
 has a maximum value, say $M'$. Since $F\in\widetilde{\Coh}_{\le 1}(X)$, we have 
 $\ch _2(F)\cdot \phi_{\ast}H'\ge 0$, where $H'$ is an ample 
 divisor on $W$. Hence $\ch_2(F)/m \in K'$, which implies 
 \begin{align}\label{ine1}
 \frac{\lvert\{(-\delta H-B)+i(f^{\ast}\omega '-\omega)\}\ch_2(F)\rvert}{\lvert Z(F)\rvert} \le M',
 \end{align}
 for all $F\in \widetilde{\Coh}_{\le 1}(X)$. 
 
 Next let us take a non-zero $G\in \oPPer(\dD_X)$ supported 
 on $C$. Since $G$ is generated by (\ref{gen1}), we can write 
 $[G]=\sum _{i=1}^n a_i [S_i]$
 in $K(\dD_X)$ for $a_i \ge 0$.  
 Let us set 
  $c_i \cneq \delta H\cdot C_i$ and 
  $c_i'\cneq \lvert ((-\delta H-B)-i\omega)\cdot C_i\rvert$.
  We have 
  \begin{align}\label{ine2}
 \frac{\lvert\{(-\delta H-B)+i(f^{\ast}\omega '-\omega)\}\ch_2(G)\rvert}{\lvert Z(G)\rvert}\le 
 \frac{\sum _{i=1}^n a_i c_i'}{a_0+ \sum_{i=1}^n a_i c_i}.
 \end{align}
Since $c_i>0$ and $a_i >0$ for some $i$, we have 
RHS$\le M''$ for some $M''>0$ independent of $a_i$. We may 
take $M''=M'$. 

Finally since $\oPPer(\dD_X)$ is generated by (\ref{gen1}), 
(\ref{gen2}),  
any $E\in \oPPer(\dD_X)$ is written as $[E]=[F]+[G]$ in $K(\dD_X)$, where
 $F\in \widetilde{\Coh}_{\le 1}(X)$ and $[G]=\sum _{i=1}^n a_i [S_i]$ for $a_i \ge 0$. 
 We have 
 \begin{align}\label{go1}
 (\ref{go}) & \le \sup \left\{ M' \cdot \frac{\lvert Z(F) \rvert 
 +\lvert Z(G) \rvert}{\lvert Z(F)+Z(G) \rvert} :
 E\in \oPPer(\dD_X) \right\} \\
 &\le M' \cdot \sup \left\{ \frac{\lvert z \rvert +1}{\lvert z+1 \rvert}
 : \Imm z \ge 1 \right\} < \infty.\end{align}
 Now we have proved 
$\lVert Z -Z_{(B,\omega)}\rVert _{\sigma}
<\infty$. As in the proof of Lemma~\ref{op}, for 
any $\varepsilon >0$ there is $B+i\omega \in A(X)_{\mathbb{C}}$
and a stability condition $\tau =(Z_{(B,\omega)}, \qQ)$
such that 
$$d(\sigma, \tau)<\varepsilon, \quad 
\lVert Z -Z_{(B,\omega)}\rVert _{\sigma}
<\sin \pi \epsilon.$$
If we show $\tau \in U_X$, we can conclude $\sigma \in \overline{U}_X$. 
The same proof of the last part of Lemma~\ref{op} shows
 $\tau \in U_X$, (it is enough to notice that in the sequence (\ref{tr}), 
 one has $H^{-1}(E)[1]\in \pP([1,b))$ and the rest is the same,)
 and we leave the detail to the reader.

\end{proof}

\subsection*{Acknowledgement}
This paper was written while the author
was visiting to the Max-Planck Institut f$\ddot{\textrm{u}}$r
Mathematik in Bonn. 
He thanks the institute for the hospitality. 
Also he thanks Atsushi Takahashi for 
valuable discussions on Gopakumar-Vafa invariants, 
Yukiko Konishi for pointing out some references, 
and Hokuto Uehara, So Okada for giving him nice comments on
 the manuscript. Finally he thanks 
 the referee for many suggestions and comments 
 for the improvement of this paper. 
He is supported by 
 Japan Society for the Promotion of Sciences Research 
Fellowships for Young Scientists, No 198007.

Yukinobu Toda, Graduate School of Mathematical Sciences, University of Tokyo

\textit{E-mail address}:toda@ms.u-tokyo.ac.jp

\end{document}